\documentclass[a4paper,11pt,english]{amsart}

\def\margin_comment#1{\marginpar{\sffamily{\tiny #1\par}\normalfont}}
\usepackage[table]{xcolor}
\usepackage{amssymb,euscript}
\usepackage{amsfonts}
\usepackage{arydshln}
\usepackage{dutchcal}
\usepackage{marginnote}

\usepackage{graphicx}
\usepackage{setspace}
\usepackage{color}
\usepackage{amsmath,amssymb,amsthm}
\usepackage{epsfig}
\usepackage{babel}
\usepackage{tikz}
 \usepackage{amsmath}
\usepackage{tikz-cd}
\usetikzlibrary{matrix, calc, arrows}
\usepackage{tikz}
\usetikzlibrary{arrows,chains,matrix,positioning,scopes}
\usepackage{lipsum}
\usepackage{mathtools}
\usepackage{mathrsfs}
\usepackage{graphicx}
\usepackage{graphics}
\usepackage{setspace}
	\usepackage{color}

\usetikzlibrary{trees}
\usetikzlibrary{arrows,shapes,automata,backgrounds,petri,positioning,scopes,matrix, calc, plothandlers,plotmarks,patterns}
\usepackage{enumerate}
\usepackage{float}
\usepackage{tikz-3dplot}
\usepackage{pgfplots}
\usepackage{tkz-graph}
\usepackage[top=1.5cm, bottom=1.5cm, outer=5cm, inner=2cm, heightrounded, marginparwidth=2.5cm, marginparsep=2cm]{geometry}

\makeatother

\newcommand{\scalemath}[2]{\scalebox{#1}{\begin{math} {#2} \end{math}}}

\newcommand*{\mat}{\mathbf}
\newcommand{\bigzero}{\mbox{\normalfont\Large\bfseries 0}}
\newcommand{\rvline}{\hspace*{-\arraycolsep}\vline\hspace*{-\arraycolsep}}
\newgeometry{left=3cm,right=3cm,top=3.5cm,bottom=3cm}

\GraphInit[vstyle = Normal]
\tikzset{
  LabelStyle/.style = {minimum width = 2em, 
                        text = red, font = \bfseries },
  VertexStyle/.append style = { inner sep=2pt,
                                font = \Large\bfseries, fill},
  EdgeStyle/.append style = {->, bend left} }

\makeatletter
\tikzset{join/.code=\tikzset{after node path={%
\ifx\tikzchainprevious\pgfutil@empty\else(\tikzchainprevious)%
edge[every join]#1(\tikzchaincurrent)\fi}}}
\makeatother
\tikzset{>=stealth',every on chain/.append style={join},
         every join/.style={->}}
\tikzstyle{labeled}=[execute at begin node=$\scriptstyle,
   execute at end node=$]
\newtheorem{thm}{Theorem}[section]
\numberwithin{equation}{section} 
\numberwithin{figure}{section} 
\theoremstyle{plain}
\newtheorem*{thm*}{Theorem}
\theoremstyle{definition}
\theoremstyle{plain}
\newtheorem{thmx}{Theorem}
 
\newtheorem*{defn*}{Definition}

\theoremstyle{plain}

\theoremstyle{plain} 

\theoremstyle{plain}

\newtheorem{prop}[thm]{Proposition} 
\theoremstyle{remark}
\newtheorem{ex}[thm]{Example}
\theoremstyle{remark}
\newtheorem{rem}[thm]{Remark}
\theoremstyle{plain}

\theoremstyle{plain}

\theoremstyle{plain}
\newtheorem{lem}[thm]{Lemma} 
\theoremstyle{definition}
\newtheorem{defn}[thm]{Definition}

\newtheorem*{acknowledgment*}{Addentum}

\theoremstyle{plain}
\newtheorem*{ex*}{Example}
\theoremstyle{plain}

\begin{document}
\title[The  Yang-Baxter equation, quantum computing and quantum entanglement]{The  Yang-Baxter equation, quantum computing and quantum entanglement}
\author{Fabienne Chouraqui}
\begin{abstract}
	
	We present a method to construct infinite families of entangling (and primitive) $2$-qudit gates, and amongst them   entangling  (and primitive) $2$-qudit gates which satisfy the Yang-Baxter equation.   We show that, given  $2$-qudit gates $c$ and $d$, if $c$ or $d$ is  entangling, then their Tracy-Singh product  $c \boxtimes d$  is also entangling and we can provide decomposable states which become entangled  after the application of $c \boxtimes d$.
\end{abstract}
\maketitle
\section*{Introduction}
The Yang-Baxter equation is an equation in mathematical physics and it lies in the  foundation of  the theory of quantum groups. It is a key concept in the construction of knots and links invariants. One of the fundamental problems is to find all the solutions of this equation.  Let $V$ be   a finite-dimensional vector space over the  field $\mathbb{C}$.  The linear operator  $c:V \otimes V \rightarrow V \otimes V$    is a solution of the  Yang-Baxter equation (YBE) if it  satisfies  the equality  $c^{12}c^{23}c^{12}=c^{23}c^{12}c^{23}$ in $V \otimes V \otimes V$, $c$ is also called a braiding operator or an $R$-matrix.  The operator  $c$ is not necessarily unitary and even the problem of  finding only  the unitary solutions of the YBE is a hard one.  A classification of all the unitary solutions of the YBE  in  case the dimension of the vector space $V$ is two is done in  \cite{dye}.  L.H. Kaufman and S.J. Lomonaco Jr  initiated in a series of papers the study of  the connections   between quantum computation, quantum entanglement and  topological entanglement, and in particular the role of unitary braiding operators in quantum computing \cite{kauf-lo1,kauf-lo1',kauf-lo4,kauf-lo3}. 

Quantum computing is a cutting-edge paradigm in the field of information processing that leverages the principles of quantum mechanics to perform computations. Unlike classical computers that rely on bits to represent information as either 0 or 1, quantum computers harness the principles of quantum mechanics to utilize quantum bits or qubits.  Qubits can exist in multiple states simultaneously, thanks to the phenomenon of superposition, allowing quantum computers to explore multiple solutions to a problem in parallel, potentially solving certain computational challenges exponentially faster than classical computers. Furthermore, quantum computers exploit the concept of entanglement.

Quantum entanglement has become  a subject of intense study and experimentation, and its properties have important  applications in quantum information processing, quantum cryptography, and quantum communication. It is a fundamental phenomenon in quantum mechanics where two or more particles become intricately correlated in such a way that the state of one particle instantaneously influences the state of another, regardless of the physical distance separating them. When particles,  such as electrons or photons,  become entangled, the measurement of the state of one particle (such as its spin, polarization, or other quantum properties) is correlated with the state of the other particle, even if they are spatially separated by large distances. Quantum entanglement has profound implications for the field of quantum mechanics. 

Quantum mechanics relies heavily on   mathematical tools to describe the state, evolution, and measurement of quantum systems.  A quantum system is represented by a state vector  (or wavefunction) which belongs to a complex Hilbert space $\mathcal{H}$, and is denoted by $\mid \psi \rangle$ in the Dirac notation.  It  encodes the probabilities of finding a particle in different states, where the probabilities are determined by the coefficients of the state vector. The  evolution is described by the  application of  an unitary operator $A$  of  $\mathcal{H}$  on  a  state vector $\mid \psi \rangle$,  and the result is a new state vector denoted by $A\mid \psi \rangle$.  The act of measurement  fundamentally alters the state of a quantum system. Indeed, upon measurement, the quantum system undergoes a sudden and non-reversible change. This process is  called  the collapse of the wavefunction and is the subject of ongoing debate among physicists. 

In the context of quantum mechanics, an entangling operator is an operator that acts on the quantum states of two or more particles, leading to the creation of entanglement between them. Entangling operators are fundamental tools for building quantum algorithms and performing quantum computations. In this paper, we propose a method to construct infinite families of entangling operators which are also solutions to the YBE. In order to be able to present our results, we need to give some formal definitions. 

A \emph{qudit} is a  state vector in the Hilbert space $\mathbb{C}^d$, where $d \geq 2$. Whenever $d=2$, it is called a \emph{qubit}. A \emph{$n$-qudit} is a state vector in the Hilbert space $(\mathbb{C}^d)^{\otimes n}$ and a \emph{$n$-qubit} is a state vector in the Hilbert space $(\mathbb{C}^2)^{\otimes n}$.  A $n$-qudit $\mid \phi \rangle$ is   \emph{decomposable} if $\mid \phi \rangle\,=\,\mid \phi_1 \rangle\,\otimes\,\mid \phi_2 \rangle\,\otimes\,...\otimes\,\mid \phi_n \rangle$, where 
$\mid \phi_i \rangle\,\in \, \mathbb{C}^d$, for $ 1 \leq i \leq n$.  Otherwise, $\mid \phi \rangle$ is  \emph{entangled}.  A \emph{(quantum)  $n$-qudits gate} is a unitary operator  $L:\,(\mathbb{C}^d)^{\otimes n}\,\rightarrow\,(\mathbb{C}^d)^{\otimes n}$. These gates belong to the unitary group $\operatorname{U}((\mathbb{C}^d)^{\otimes n})$. A sequence $L_1,L_2,...,L_k$ of quantum gates constitutes a \emph{quantum circuit} on $n$-qudits, which output is the product gate  
$L_1\cdot L_2\cdot...\cdot L_k$. In practice, one wants to build circuits out of gates which are local in the sense that they operate on a small number of qudits, typically $1$, $2$, or $3$ \cite{bryl}. 

A $2$-qudit gate $L:\,(\mathbb{C}^d)^{\otimes 2}\,\rightarrow\,(\mathbb{C}^d)^{\otimes 2}$ is \emph{primitive} if $L$ maps decomposable $2$-qudit to decomposable $2$-qudit, otherwise $L$ is  said to be \emph{imprimitive} in  \cite{bryl} or \emph{entangling} in \cite{kauf-lo3}. That is, the gate   $L$ is said to be  entangling, if there exists a decomposable $2$-qudit $\mid \phi \rangle$ such that $L\,\mid \phi \rangle$ is entangled. In \cite{bryl}, J.L. Brylinski and R. Brylinski give a criteria to decide whether a  $2$-qudit  gate  $L$ is primitive or entangling and moreover, they prove that $L$ is entangling if and only if  $L$ is exactly universal, which means  that the collection of all $1$-qudit gates together with $L$ generates the unitary group  $\operatorname{U}((\mathbb{C}^2)^{\otimes n})$.

 The Yang-Baxter equation plays a significant role in quantum computing, particularly in the context of quantum information theory and in the domain of  integrable quantum computation,  a type of quantum circuit model of computation in which two-qubit gates are either the swap gate or non-trivial unitary solutions of the Yang-Baxter equation \cite{integrable-quantum}.
 Recently, the Yang-Baxter equation has been used in the study of quantum error correction   \cite{ybe-quantum-correction}, and in teleportation-based quantum computation \cite{teleportation}.  There is also an increasing interest on the realisation and the implementation of  Yang–Baxter gates and the Yang–Baxter equation on quantum computers. The Yang–Baxter gates from the integrable circuit have been implemented on IBM superconducting-based quantum computers \cite{imp1,imp2}, and  the Yang–Baxter equation has been tested on the optical and NMR systems \cite{test1,test2}. In \cite{implement}, the authors address the question of how to find  optimal  realisations of Yang-Baxter gates. 
 
  Every  unitary solution of  the Yang-Baxter equation, $R:V \otimes V \rightarrow V \otimes V$, where $V$ is a complex vector space of dimension $d$, can be described by $R:\,\mathbb{C}^d\otimes\mathbb{C}^d\,\rightarrow\,\mathbb{C}^d\otimes\mathbb{C}^d$ and $R$  is a $2$-qudit gate.   In \cite{kauf-lo3}, L.H. Kaufman and S.J. Lomonaco Jr  initiate the study of $2$-qubit gates, $R:\,\mathbb{C}^2\otimes\mathbb{C}^2\,\rightarrow\,\mathbb{C}^2\otimes\mathbb{C}^2$, which  satisfy  the Yang-Baxter equation.  In particular, using the classification of unitary solutions from \cite{dye}, they investigate this special class of  $2$-qubit gates and determine which are entangling using the criteria from \cite{bryl}. The question whether   a $2$-qudit gate $R$  that satisfy the Yang-Baxter equation is  primitive  or entangling   has found an interesting application in the domain of  link invariants.  Indeed, in  \cite{entangle-knot}, the authors prove that  if  such   a  $2$-qudit gate $R$  is  primitive (non-entangling), then the  link invariant obtained form $R$,  using the Turaev construction  \cite{turaev}, is trivial.\\

 In this paper,  we  provide a tool to construct infinitely many  entangling  and primitive  $2$-qudit  gates and in particular relying on the results from \cite{bryl} and from \cite{kauf-lo3}, we construct  entangling  and primitive   $2$-qudit gates that  satisfy  the YBE. 
 \begin{thmx} \label{thm0}
	Let  $d\geq 2$ be any  integer.
	Then,  there exists an entangling  $2$-qudit gate  $R:\,\mathbb{C}^d\otimes\mathbb{C}^d\,\rightarrow\,\mathbb{C}^d\otimes\mathbb{C}^d$  and there exists a primitive  $2$-qudit gate  $S:\,\mathbb{C}^d\otimes\mathbb{C}^d\,\rightarrow\,\mathbb{C}^d\otimes\mathbb{C}^d$, where  $R$ and $S$ satisfy  the Yang-Baxter equation.
\end{thmx}  
In \cite{pourkia-annals-phys}, A. Pourkia and J. Batle  construct an infinite family of unitary and entangling  $2$-qubit gates that satisfy  the Yang-Baxter equation,  using the non-trivial quasi-triangular structure on the group algebra of the cyclic group of order $n\geq 2$ from \cite{majid-book}. A. Pourkia generalises  the construction of  \cite{pourkia-annals-phys} to  construct entangling $2$-qudit gates, since the physical implementation of qudit-based quantum computing systems has several advantages, such as larger storing and processing space or information, a smaller number of qudits required to span the space state, better noise resistance, simplified quantum logic and more  \cite{pourkia}. In \cite{rowell}, generalised solutions of the YBE have been defined to deal with multiple qubit gates, leading to different definitions of entanglement, and in \cite{padma}, the authors construct large families of the generalised YBE and they show that for every multi-qubit  entangled  state in a particular class there is a canonical way to construct the unitary $R$-matrices that solve the  generalised YBE  and generate that state.\\

The steps in the construction of   entangling  $2$-qudit gates are the following.  As a first step, we show that, given two (unitary)  solutions of the  YBE,  their Tracy-Singh product is also a (unitary)  solution of the  YBE. Applying this process iteratively enables  the construction of  infinite families of  solutions of the YBE, and in particular of  unitary solutions or  $2$-qudit gates. More precisely we prove:
\begin{thmx} \label{thm1}
	  Let $c: \mathbb{C}^n \otimes \mathbb{C}^n \rightarrow \mathbb{C}^n \otimes \mathbb{C}^n $ and $d:\mathbb{C}^m \otimes \mathbb{C}^m \rightarrow \mathbb{C}^m \otimes \mathbb{C}^m$ be   solutions of 	the  Yang-Baxter equation, with matrices $c$ and $d$ with respect to their standard  bases respectively.
	Let $c \boxtimes d$ denote their  Tracy-Singh product,  with a   specific block partition. 	Then, the linear operator  $\tilde{c_d}:\mathbb{C}^{nm} \otimes \mathbb{C}^{nm} \rightarrow \mathbb{C}^{nm} \otimes \mathbb{C}^{nm} $ with representing matrix $c\boxtimes d$  is a solution of the Yang-Baxter equation.  Furthermore, if $c$ and $d$  are unitary, then $\tilde{c_d}$ is also an unitary operator.
\end{thmx}  
The Tracy-Singh product of matrices is a  generalisation of the  tensor or Kronecker product,   called sometimes the block Kronecker product, as  it requires a block partition of the matrices.  We present  a connection between  the Tracy-Singh product  and a categorical construction. We find that   in the same way  the   monoidal category of finite dimensional vector spaces,  $\operatorname{Vec}$, is endowed with the tensor product $\otimes$ as a monoidal product, there exists   a  monoidal category, we call $\mathcal{Diag}$,  that can be endowed with $\boxtimes$ as its monoidal product. The  category  $\mathcal{Diag}$ is a particular subcategory of  $\operatorname{Vec}\,\otimes\,\operatorname{Vec}$, with the crucial property that the representing matrices of morphisms in $\mathcal{Diag}$ admit  a  uniquely defined  partition into blocks all of the same size.

As a second step, we show that  if  $c$ or  $d$ is a   $2$-qudit entangling  gate, then their Tracy-Singh product  $c \boxtimes d$  is also entangling and we can provide decomposable state vectors  which become entangled  after the application of $c \boxtimes d$. That is, we prove:
\begin{thmx}\label{thm2}
	Let $c: \mathbb{C}^n \otimes \mathbb{C}^n \rightarrow \mathbb{C}^n \otimes \mathbb{C}^n $ and $d:\mathbb{C}^m \otimes \mathbb{C}^m \rightarrow \mathbb{C}^m \otimes \mathbb{C}^m$ be  $2$-qudit gates. 	Let $c \boxtimes d$ denote their  Tracy-Singh product,  with a   specific block partition. 	 If  either $c$ or  $d$  is an entangling $2$-qudit gate, then  $c \boxtimes d$ is also an entangling $2$-qudit gate. Also, if  both $c$ or  $d$  are primitive  $2$-qudit gates of a certain kind, then  $c \boxtimes d$ is also a primitive $2$-qudit gate (of the same kind). 
\end{thmx}  
The paper is divided in two parts. Roughly, the first part presents  the  construction of solutions of the YBE  using the Tracy-Singh product,  and the second part presents the topic of quantum computing, the construction of primitive and entangling gates using the Tracy-Singh product. More in detail,  in Section $1$, we give some preliminaries on  solutions of the YBE, in Section $2$, we give some preliminaries on matrix products and their properties, and  in Section $3$,   we   present the properties of a solution of the YBE, describe  the method of construction of large solutions and prove Theorem \ref{thm1}.  In Section $4$, we present some connections between the Tracy-Singh product and other constructions, and in particular a connection with a categorical construction. In Section $5$, we  give some preliminaries on quantum computing and entanglement, in Section $6$, we prove Theorems \ref{thm2} and  \ref{thm0}. 
 The author confirms that the data supporting the findings of this study are available within the article.
 \part{The   Yang-Baxter equation  (YBE) and construction of  an infinite family of new solutions using the Tracy-Singh product}
 \section{Preliminaries on  solutions of the   Yang-Baxter equation  (YBE) }
 We use the terminology from \cite[Ch.VIII]{kassel}.  We refer the reader to  \cite{kassel} for more details.
 \begin{defn}
 	Let $V$ be a vector space over a field $k$. A linear automorphism $c$ of  $V \otimes V$ is said to be an $R$-matrix if it is a solution of the Yang-Baxter equation (YBE)
 	\begin{equation}\label{eqn-ybe}
 	(c \otimes Id_V)(Id_V \otimes c )(c \otimes Id_V)\,=\,(Id_V \otimes c)(c \otimes Id_V)(Id_V\otimes c)
 	\end{equation}
 	that holds in the automorphism group of  $V \otimes V \otimes V$.
 \end{defn}
 It is also written as $c^{12}c^{23}c^{12}=c^{23}c^{12}c^{23}$. Let $\{e_i\}_{i=1}^{i=n}$ be a basis of the finite dimensional vector space $V$.  An automorphism $c$ of   $V \otimes V$ is defined by the family $(c_{ij}^{kl})_{i,j,k,l}$ of scalars determined by $c(e_i\otimes e_j)\,=\,\sum\limits_{k,l} \,c_{ij}^{kl}\,e_k \otimes e_l$.  Then $c$ is an $R$-matrix if and only if  its representing  matrix,  $c$, with respect to  the basis  $\{e_i\otimes e_j\mid 1\leq i,j\leq n\}$,  satisfies
 \begin{equation}\label{eqn-ybe-matrix}
 (c \otimes  \mat{I}_{n})\,(\mat{I}_{n}\otimes  c)\,(c \otimes  \mat{I}_{n})\,=\,(\mat{I}_{n}\otimes  c)\,(c \otimes  \mat{I}_{n})\,(\mat{I}_{n}\otimes  c)
 \end{equation}
 which is equivalent, for all $i,j,k,l,m,n$,  to $\sum\limits_{p,q,y,}\,c_{ij}^{pq}
 c_{qk}^{yn}\,
 c_{py}^{lm}\,=\,
 \sum\limits_{y,q,r}\,
 \,c_{jk}^{qr}\,
 c_{iq}^{ly}\,
 c_{yr}^{mn} $. \\Solving these  non-linear equations  is a highly non-trivial problem. Nevertheless, numerous  solutions of the YBE have been found.  
 \begin{defn}\label{def-iso}
 	Let $c: V \otimes V \rightarrow V \otimes V$ and $c': V' \otimes V' \rightarrow V' \otimes V'$
 	be solutions of Yang-Baxter equation, where $V$ and $V'$ are vector spaces over the same field $k$, with $\operatorname{dim}(V)=\operatorname{dim}(V')=n$. The solutions $c$ and  $c'$ are \emph{isomorphic} if there exists a linear isomorphism $\mu:V \rightarrow V'$ such that 
 	$c'\,(\mu \otimes\mu)\,=\,(\mu \otimes\mu)\,c$.
 \end{defn}
 \begin{ex}\label{ex1-kaufman}(\cite{kauf-lo3}, \cite{dye} and \cite{kassel})
 	Let $c,\,d:V \otimes V \rightarrow V \otimes V$ be  $R$-matrices, with $\operatorname{dim}(V)=2$
 	\[c=	\begin{pmatrix}
 	\frac{1}{\sqrt{2}}& 0 & 0 & 	\frac{1}{\sqrt{2}}\\
 	0 &	\frac{1}{\sqrt{2}} & -	\frac{1}{\sqrt{2}} & 0 \\
 	0 &	\frac{1}{\sqrt{2}} & 	\frac{1}{\sqrt{2}} &0 \\
 	-	\frac{1}{\sqrt{2}} &0 & 0 &	\frac{1}{\sqrt{2}} \\
 	\end{pmatrix}\;\;\;  \textrm{and}\;\;\; 
 	d=\begin{pmatrix}
 	2& 0 & 0 & 	0\\
 	0 &0 & 1 & 0 \\
 	0 &	1 & 1.5&0 \\
 	0&0 & 0 &	2\\
 	\end{pmatrix}\]
 	Here, we have $c(e_1 \otimes e_2)=\,\frac{1}{\sqrt{2}} \,e_1 \otimes e_2\,+\,\frac{1}{\sqrt{2}} \,e_2 \otimes e_1$ and $d(e_1 \otimes e_2)=\,e_2 \otimes e_1$. Note that $c$ is unitary, while $d$ is not. As a convention, we  always consider the  basis   $\{e_i \otimes e_j \,\mid \,1\leq i,j \leq  n \}$ of $V\otimes V$ ordered lexicographically, that is, as an example, for $n=2$, the ordered basis of $V\otimes V$  is 
 	$\{e_1\otimes e_1 ,\,e_1\otimes e_2,\,e_2\otimes e_1,\,e_2\otimes e_2 \}$.
 \end{ex}
 There  is a  way to generate new $R$-matrices from old ones. Indeed,  if $c \in \operatorname{Aut}(V \otimes V)$ is an $R$-matrix, then so are $\lambda c$, $c^{-1}$ and $\tau \circ c \circ \tau$, where $\lambda$ is any non-zero scalar and $\tau: V \otimes V$ is the flip map ($\tau(v_1 \otimes v_2)=v_2\otimes v_1$) \cite[p.168]{kassel}.

 In \cite{drinf}, Drinfeld suggested the study of a particular class of solutions,  derived from the so-called set-theoretic solutions. The study of these solutions was initiated in \cite{etingof}. We refer also to  \cite{gateva_van}  and \cite{jespers_book}. Let $X$ be a non-empty set. Let $r: X \times X \rightarrow X \times X$  be a map and write $r(x,y)=(\sigma_{x}(y),\gamma_{y}(x))$,  where $\sigma_x, \gamma_y:X\to X$ are functions  for all  $x,y \in X$.   The pair $(X,r)$ is  \emph{braided} if $r^{12}r^{23}r^{12}=r^{23}r^{12}r^{23}$, where the map $r^{ii+1}$ means $r$ acting on the $i$-th and $(i+1)$-th components of $X^3$.  In this case, we  call  $(X,r)$  \emph{a set-theoretic solution of the  Yang-Baxter equation or st-YBE }, and whenever $X$ is finite, we  call  $(X,r)$  \emph{a finite st-YBE}.  The pair $(X,r)$ is \emph{non-degenerate} if for every  $x\in X$,  $\sigma_{x}$ and $\gamma_{x}$  are bijective and it   is  \emph{involutive} if $r\circ r = Id_{X^2}$. If $(X,r)$ is a non-degenerate involutive  st-YBE, then $r(x,y)$ can be described as  $r(x,y)=(\sigma_{x}(y),\gamma_{y}(x))=(\sigma_{x}(y),\,\sigma^{-1}_{\sigma_{x}(y)}(x))$.  A st-YBE  $(X,r)$ is \emph{square-free},  if  for every $x \in X$, $r(x,x)=(x,x)$. A st-YBE $(X,r)$ is \emph{trivial} if $\sigma_{x}=\gamma_{x}=Id_X$, for every  $x \in X$. Two  set-theoretic solutions $(X,r)$ and $(X',r')$ are \emph{isomorphic} if there is a bijection $\mu:X \rightarrow X'$ such that $(\mu \times \mu) \circ r=r'\circ (\mu \times \mu)$  \cite{etingof}.   Non-degenerate and involutive st-YBE are  intensively investigated and they give rise to several algebraic structures associated to them \cite{gigel}, \cite{catino3, brace}, \cite{chou_art, chou_godel1,chou_godel2},  \cite{doikou}, \cite{etingof,gateva_van,gateva-phys1,gateva_new,g-vendra1}, \cite{jespers_book}, \cite{lebed1}, \cite{rump_braces,smot}.
 \begin{ex}\label{ex-permut-sol-2}
 	For $X=\{x_1,x_2\}$,  there are exactly two non-degenerate and involutive st-YBE. The first one is called  a trivial solution with $\sigma_1=\sigma_2=\gamma_1=\gamma_2=Id_X$  and the second one is called a permutation solution with  $\sigma_1=\sigma_2=\gamma_1=\gamma_2=(1,2)$.  Their matrices with respect to the ordered standard basis  $\{e_i \otimes e_j \mid 1 \leq i,j \leq 2\}$ are respectively 	$c_{2,1}=\begin{pmatrix}
 	1& 0 & 0 & 	0\\
 	0 &0 & 1 & 0 \\
 	0 &	1 & 0&0 \\
 	0&0 & 0 &	1\\
 	\end{pmatrix}\;\;\; \textrm{and}\;\;\; 	c_{2,2}=\begin{pmatrix}
 	0& 0 & 0 & 	1\\
 	0 &1 & 0 & 0 \\
 	0 &	0 & 1&0 \\
 	1&0 & 0 &	0\\
 	\end{pmatrix}$.
 \end{ex}
 \begin{defn}
 	Let  $(X,r)$ be a non-degenerate  involutive st-YBE. A set $Y \subset X$ is  \emph{invariant} if  $r(Y \times Y)\subset Y \times Y$. An invariant  subset $Y \subset X$  is  \emph{non-degenerate} if  $(Y,r\mid_{Y^2})$ is a non-degenerate involutive st-YBE. A st-YBE.  $(X,r)$  is \emph{decomposable} if it is a union of two non-empty disjoint non-degenerate invariant subsets. Otherwise, it is  \emph{indecomposable}.
 \end{defn}
 A very simple class of  non-degenerate involutive st-YBE is the class of  \emph{permutation solutions}. These solutions have the form $r(x,y)=(\sigma(y),\sigma^{-1}(x))$,  where the bijections  $\sigma_x: X\to X$  are  all equal and equal to $\sigma$, the   bijections  $\gamma_x: X\to X$  are  all equal and equal to $\sigma^{-1}$. If $\sigma$ is a cyclic permutation, $(X,r)$ is a \emph{cyclic permutation solution}. A permutation solution is indecomposable if and only if it is cyclic \cite[p.184]{etingof}.
 \begin{prop}\cite{etingof2}\label{prop-indecomp-prime}
 	For every prime $p$, there exists a unique indecomposable non-degenerate involutive st-YBE (up to iso), the cyclic permutation solution with $\sigma=(1,2,...,p)$.
 \end{prop}
 \begin{ex}\label{ex-square-free-sol-3}
 	For $X=\{x_1,x_2,x_3\}$,  there are five non-degenerate and involutive st-YBE. From the non-degeneracy and the  involutivity of a  solution, $r$ is a permutation matrix for a permutation which is the product of disjoint transpositions. Indeed, for every $1 \leq i,j,k,l \leq3$,  if $r(e_i\otimes e_j)=e_k\otimes e_l$, then $r(e_k\otimes e_l)=e_i\otimes e_j$. 
 	\begin{enumerate}[(i)]
 		\item  One of these  solutions is the cyclic permutation solution defined by  $\sigma=(1,2,3)$ and $\gamma=\sigma^{-1}=(1,3,2)$. Its matrix  $s$ is a square matrix of size 9, which satisfies $s(e_1 \otimes e_1)=e_2 \otimes e_3,\,s(e_1 \otimes e_2)=e_3 \otimes e_3, \,...,\,s(e_3 \otimes e_1)=e_2 \otimes e_2$, and  $e_1 \otimes e_3, \, e_2 \otimes e_1,\, e_3\otimes e_2$ are fixed. One can describe the square matrix  $s$, with the following partition into square blocks of size 3, where $E_{ij}$ denotes the elementary matrix with $1$ at position $ij$, and $0$ elsewhere, as in Equation (\ref{eqn-s-cyclic-3}): 
 		\begin{equation}\label{eqn-s-cyclic-3}
 		s=\;\left(\begin{array}{c:c:c}
 		E_{33}
 		& E_{13}
 		& E_{23} \\
 		\hdashline
 		E_{31}
 		& E_{11}
 		& E_{21} \\
 		\hdashline
 		E_{32}	
 		& E_{12}
 		& E_{22} \\
 		\end{array}\right)	\end{equation}
 		\item Another solution  is the square-free solution defined by  $\sigma_i=\gamma_i=Id$, $i=1,2$ and   $\sigma_3=\gamma_3=(1,2)$.  Furthermore, $r(e_j \otimes e_j)=e_j\otimes e_j$, $1\leq j\leq 3$, since it is square-free.  The matrix $r$ is  described in Equation (\ref{eqn-r-square-free}):
 		\begin{equation}\label{eqn-r-square-free}
 		r=\;\;\;\left(\begin{array}{c:c|c}
 		E_{11}
 		& E_{21}
 		& E_{32} \\
 		\hdashline
 		E_{12}
 		& E_{22}
 		& E_{31} \\
 		\hline
 		E_{23}	
 		& E_{13}
 		& E_{33} \\
 		\end{array}\right)
 		\end{equation}
 	\end{enumerate}
 \end{ex}
\section{Preliminaries on some products of matrices and their properties} 
We refer to \cite{hyland},   \cite{liu}, \cite{commut}, \cite{tracy}, \cite{tracy-jina}, \cite{zhour}  for more details.
Let $A=(a_{ij})$ be a matrix  of size $m \times n$ and $B=(b_{kl})$ of size $p \times q$. Let $A=(A_{ij})$ be partitioned with $A_{ij}$ of size $m_i \times n_j$ as the $ij$-th block submatrix and let $B=(B_{kl})$ be partitioned with $B_{kl}$ of size $p_k \times q_l$ as the $kl$-th block submatrix ($\sum m_i=m,\,\sum n_j=n,\,\sum p_k=p,\,\sum q_l=q$). The Kronecker (or tensor) product and the Tracy-Singh (or block Kronecker)  product are defined as follows:
\begin{enumerate}
	\item  The Kronecker (or tensor) product:
	\begin{equation*}
A\,\otimes \, B\,=\,(a_{ij}B)_{ij}
	\end{equation*}
	The matrix $A\,\otimes \, B$ is of size $mp \times nq$ and the block $a_{ij}B$ is size $p \times q$.
\item  The Tracy-Singh (or block Kronecker)  product:
	\begin{equation*}
	A\,\boxtimes \, B\,=\,((A_{ij} \otimes B_{kl})_{kl})_{ij}
	\end{equation*}
	The matrix $A\,\boxtimes \, B$ is of size $mp \times nq$ and the block $A_{ij} \otimes B_{kl}$ is size $m_ip \times n_jq$.
	
	The matrix $A\,\star \, B$ is of size $(\sum m_ip_i )\times (\sum n_jq_j )$ and the block $A_{ij} \otimes B_{ij}$ is size $m_ip_i\times n_jq_j$.
\end{enumerate}
For non-partitioned matrices,  $A\,\boxtimes \, B\,=\, A\,\otimes \, B$.
\begin{ex}\label{ex-tracy-c-d}
	Let $c$ and $d$ the square matrices of size $4$ from Example \ref{ex1-kaufman}, with a partition into $4$ square blocks of size 2. Let $B_{ij}$, and  $B'_{ij}$, $1 \leq i,j \leq 2$, denote the blocks of $c$ and $d$ respectively (see Example \ref{ex2-kaufman}).   Then $c\boxtimes d$ is a square matrix of size 16, which can be partitioned into 16 square blocks of size 4. The square block of size 4 at position  $22$ is obtained from $B_{11}\otimes B'_{22}$ and the block  at position  $24$ is obtained from $B_{12}\otimes B'_{22}$:  \[B_{11}\otimes B'_{22}\;=\;\begin{pmatrix}
		\frac{1.5}{\sqrt{2}}& 0 & 0 & 	0\\
		0 &\sqrt{2} & 0 & 0 \\
		0 &	0 & 	\frac{1.5}{\sqrt{2}}&0 \\
		0&0 & 0 &		\sqrt{2}\\
	\end{pmatrix}\;\;\; ;\;\;\; 
	 B_{12}\otimes B'_{22}\;=\;\begin{pmatrix}
0& 0 & \frac{1.5}{\sqrt{2}}& 	0\\
	0 &	0& 0 & \sqrt{2}  \\
	-\frac{1.5}{\sqrt{2}} &	0& 	0&0 \\
	0&-\sqrt{2} & 0 &		0\\
	\end{pmatrix}\]
\end{ex}
In the following Theorems, we list  important properties of the Tracy-Singh product.
\begin{thm}\cite{tracy}\label{thm-tracy}
		Let $A$, $B$, $C$, and $D$  be matrices. Then
\begin{enumerate}[(i)]
	\item $A \boxtimes B$ and $B \boxtimes A$ exist for any  matrices $A$ and $B$.
	\item $A \boxtimes B \neq B \boxtimes A$ in general.
	\item  $(A \boxtimes B)\,\boxtimes C= \,A\boxtimes \,(B \boxtimes C)$.
		\item $(A+ B)\, \boxtimes\,(C +D)=\, A \boxtimes C+A \boxtimes D +B \boxtimes C+B \boxtimes D$,  if $A+B$ and $C+D$ exist.
	\item $(A \boxtimes B)\,(C \boxtimes D)=\, AC \boxtimes BD$,  if $AC$ and $BD$ exist.
	\item $(cA) \boxtimes B\,=\,c (A\boxtimes B \,=\,A \boxtimes (cB)$.
	\item  $(A \boxtimes B)^{-1}\,=  A^{-1} \boxtimes B^{-1}$, if $A$ and $B$ are invertible.
		\item  $(A \boxtimes B)^{t}\,=  A^{t} \boxtimes B^{t}$.
		\item 	$\mat{I}_{n} \boxtimes \mat{I}_{m}\,=\,\mat{I}_{nm}$ for identity partitioned matrices.
\end{enumerate}
\end{thm}
Note that for the Kronecker product of matrices, no partition into blocks is needed. \\
Yet, if 
$A=
\begin{pmatrix}
B_{11}

& \rvline & B_{12}
& \rvline &... &\rvline & B_{1s} \\

\hline
...

& \rvline & ...
& \rvline &... &\rvline & ... \\
\hline

B_{n1}

& \rvline & B_{n2}
& \rvline &... &\rvline & B_{ns} \\
\end{pmatrix}$, then $A\otimes B= \begin{pmatrix}
B_{11} \otimes B

& \rvline & B_{12}\otimes B

& \rvline &... &\rvline & B_{1s}\otimes B
 \\
\hline
...

& \rvline & ...
& \rvline &... &\rvline & ... \\
\hline

B_{n1}\otimes B

& \rvline & B_{n2}\otimes B

& \rvline &... &\rvline & B_{ns} \otimes B
\\
\end{pmatrix}$. 

In matrix theory, the commutation matrix is used for transforming the vectorized form of a matrix into the vectorized form of its transpose. 
\begin{defn}\cite{commut} \label{defn-commut}
	The \emph{commutation matrix} $K_{mn}$  is the  matrix defined by:
	\begin{equation*}
	K_{mn}=\, \sum\limits_{i=1}^{i=m}\,\sum\limits_{j=1}^{j=n}E_{ij}\otimes E^t_{ij}
	\end{equation*}
	where $E_{ij}$ is a matrix of size $ m \times n$ with a $1$ in its $ij$-th position and zeroes elsewhere.\\
	In words, $K_{mn}$  is the square matrix of size $mn$, partitioned into $ mn$ blocks of size $ n \times m$ such that the $ij$-th block has a 1 in its $ ji$-th position and $0$  elsewhere, with  $K_{nm}\,=\,	K_{mn}^{-1}$.
\end{defn}
For example (see \cite[p383]{commut}), $K_{23}=\,\,\left(\begin{array}{c:c|c}
E_{11}
& E_{21}
& E_{31} \\
\hdashline
E_{12}
& E_{22}
& E_{32} \\
\end{array}\right)$, where $E_{ij}$ are of size $3 \times 2$.
\begin{thm}\cite{tracy-jina}, \cite{commut}, \cite{neud}\label{thm-tracy-K}
	Let $A$ be of size $n \times s$ and $B$ of size $m \times t$.
	Let $K_{mn}$ denote the  commutation matrix of size $mn$ as in Definition \ref{defn-commut}.  Then
	\begin{enumerate}[(i)]
\item  $B\otimes A\,=\,K_{mn}(A\otimes B)K_{st}$.
\item If  $A$ has a block partition into blocks $A_{ij}$, $1\leq i \leq p$, $1 \leq j \leq q$,  all of size $n'\times s'$,  and $B$ has a block partition into blocks  $B_{kl} $, $1 \leq k\leq u$,  $1 \leq l\leq v$, all of size $m'\times t'$: 
\begin{itemize}
\item $A\boxtimes B\,=\,  (\mat{I}_{p}\,\otimes K_{un'} \,\otimes \mat{I}_{m'})\,	\cdot\, (A\otimes B)\, \cdot\, (\mat{I}_{q}\,\otimes K_{s'v}\,\otimes  \mat{I}_{t'})$.
\item $B\boxtimes A\,=\, 	P \cdot\, (A\boxtimes B)\,\cdot\,Q$.
\end{itemize}
where $P$ and $Q$ are the following permutation matrices:
 \begin{gather*}
P=\, (\mat{I}_{u}\,\otimes K_{pm'} \,\otimes \mat{I}_{n'})\,	\cdot\, K_{mn}\,\cdot\, (\mat{I}_{p}\,\otimes K_{n'u} \,\otimes \mat{I}_{m'})\\
Q=\, (\mat{I}_{q}\,\otimes K_{vs'} \,\otimes \mat{I}_{t'})\,	\cdot\, K_{st}\,\cdot\, (\mat{I}_{v}\,\otimes K_{t'q} \,\otimes \mat{I}_{s'})
\end{gather*}
\end{enumerate}
\end{thm}
 For more general formulas, in the case that the blocks do not have necessarily the same size, we refer to \cite{neud}. From Theorem \ref{thm-tracy-K} $(ii)$ and $(iii)$, the  formulas connecting  $A\boxtimes B$ with  $B \otimes A$ and with $B \boxtimes A$  are  reminiscent to the formula of change of basis, but  in general the permutation matrices in  Theorem \ref{thm-tracy-K} $(ii)$ or  $(iii)$ are  not  the inverse  one of another.

\section{Construction of large  solutions of the YBE and proof of Theorem \ref{thm1} }
 Let $c: (\mathbb{C}^n )^{\otimes 2} \rightarrow (\mathbb{C}^n )^{\otimes 2}$ and $d:(\mathbb{C}^m )^{\otimes 2}\rightarrow (\mathbb{C}^m )^{\otimes 2}$ be   $R$-matrices.    In this section, we present a method to construct iteratively  from  $c$ and $d$ an infinite family of  solutions of 	the  YBE. More precisely, we show that, $c \boxtimes d$,   the Tracy-Singh product of  the $R$-matrices   $c$ and $d$ with a particular partition into blocks of $c$ and $d$    is  also a  $R$-matrix, denoted by  $c \boxtimes d:\,(\mathbb{C}^{nm}) ^{\otimes 2}\rightarrow (\mathbb{C}^{nm}) ^{\otimes 2}$. Before we proceed, we want to address a question  that arises  naturally, which is  why  looking at  the Tracy-Singh product of  the $R$-matrices   $c$ and $d$  and not  at  their  Kronecker  product. As  Example \ref{ex3-kaufman} illustrates it, even in the case   $d=c$, the linear automorphism  $ c \otimes c: (\mathbb{C}^n )^{\otimes 4}    \rightarrow (\mathbb{C}^n )^{\otimes 4}$  is  not necessarily a solution of the YBE.
\begin{ex}\label{ex3-kaufman}
	Let $c$ be  the $R$-matrix from Example \ref{ex1-kaufman}.
	Then,   $c\otimes c$ is the following square matrix of size $16$, where $c'$ is the same matrix as $c$ with each entry  $\frac{1}{\sqrt{2}}$ replaced by $\frac{1}{2}$,  
\scalebox{0.9}[0.7]{$c\otimes c=	\begin{pmatrix}
	c'& \mat{0}_{4}  & \mat{0}_{4} & 	c'\\
	\mat{0}_{4}&	c'	 & -	c' & \mat{0}_{4} \\
	\mat{0}_{4} &	c'	 &	c' &\mat{0}_{4} \\
	-	c'&\mat{0}_{4}  &\mat{0}_{4} &	c' \\
	\end{pmatrix}$},  $c\otimes c$ is not a R-matrix, since it does not satisfy Equation (\ref{eqn-ybe-matrix}).\\
	 Indeed, if we denote $c\otimes c\otimes \mat{I}_{4}$  by $\hat{c}^{12}$  and $\mat{I}_{4}\otimes c\otimes c$ by $\hat{c}^{23}$, then 
	$\hat{c}^{12}\hat{c}^{23}\hat{c}^{12} \,\neq \,\hat{c}^{23}\hat{c}^{12}\hat{c}^{23}$, since $(\hat{c}^{23}\hat{c}^{12}\hat{c}^{23})_{1,4}\,=-0.5$ while $(\hat{c}^{12}\hat{c}^{23}\hat{c}^{12})_{14}=0$. 
\end{ex} 

\subsection{Properties of solutions of the Yang-Baxter equation}\label{subsec-properties-ybe}
Let $c: (\mathbb{C}^n )^{\otimes 2} \rightarrow (\mathbb{C}^n )^{\otimes 2} $ be a  $R$-matrix, with  partition into  $n^2$ square  blocks  $B_{ij}$ of size $n$ as in Equation (\ref{eqn-c-blocks}).
\begin{equation}\label{eqn-c-blocks}
c=
\begin{pmatrix}
B_{11}

& \rvline & B_{12}
& \rvline &... &\rvline & B_{1n} \\
\hline
B_{21}

& \rvline & B_{22}
& \rvline &... &\rvline & B_{2n} \\
\hline
...

& \rvline & ...
& \rvline &... &\rvline & ... \\
\hline

B_{n1}

& \rvline & B_{n2}
& \rvline &... &\rvline & B_{nn} \\

\end{pmatrix}
\end{equation}
So, by definition of $\otimes$, $\mat{I}_{n} \otimes c$ and $c \otimes \mat{I}_{n}$ are  square matrices of size $n^3$, as described in Equation (\ref{eqn-I-tensor-c-blocks}), with  a partition into square  blocks of size $n^2$ for a better understanding.
\begin{equation}\label{eqn-I-tensor-c-blocks}
\scalemath{0.92}{
\mat{I}_{n} \otimes c=
\begin{pmatrix}
c

& \rvline & \bigzero_{n^2}
& \rvline &... &\rvline & \bigzero_{n^2}\\
	\hline
\bigzero_{n^2}

& \rvline & c
& \rvline &... &\rvline & \bigzero_{n^2}\\
	\hline
...

& \rvline & ...
& \rvline &... &\rvline & ... \\
	\hline

\bigzero_{n^2}

& \rvline & \bigzero_{n^2}
& \rvline &... &\rvline & c \\

\end{pmatrix}
\;\;\;\;\;\;\;\;
c\,\otimes \,\mat{I}_{n} =
\begin{pmatrix}
B_{11}\otimes \mat{I}_{n} 

& \rvline & B_{12}\otimes \mat{I}_{n}
& \rvline &... &\rvline & B_{1n} \otimes \mat{I}_{n}\\
	\hline
B_{21}\otimes \mat{I}_{n}

& \rvline & B_{22}\otimes \mat{I}_{n}
& \rvline &... &\rvline & B_{2n} \otimes \mat{I}_{n}\\
	\hline
...

& \rvline & ...
& \rvline &... &\rvline & ... \\
		\hline

B_{n1}\otimes \mat{I}_{n}

& \rvline & B_{n2}\otimes \mat{I}_{n}
& \rvline &... &\rvline & B_{nn} \otimes \mat{I}_{n}
\end{pmatrix} }
\end{equation}
\begin{ex}\label{ex2-kaufman}
For  the $R$-matrices $c$ and $d$ from Example \ref{ex1-kaufman},  their partition into blocks:   $c=	\begin{pmatrix}
	\frac{1}{\sqrt{2}}& 0 &	 \rvline & 0 & 	\frac{1}{\sqrt{2}}\\
	0 &	\frac{1}{\sqrt{2}} 	& \rvline & -	\frac{1}{\sqrt{2}} & 0 \\
	\hline
	0 &	\frac{1}{\sqrt{2}} 	& \rvline &	\frac{1}{\sqrt{2}} &0 \\
	-	\frac{1}{\sqrt{2}} &0 	& \rvline & 0 &	\frac{1}{\sqrt{2}} \\
	\end{pmatrix}$ and   $d=	\begin{pmatrix}
	2& 0 &	 \rvline & 0 & 	0\\
	0 &	0	& \rvline & 1 & 0 \\
	\hline
	0 &	1& \rvline &1.5 &0 \\
0&0 	& \rvline & 0 &	2 \\
	\end{pmatrix}$.
\end{ex}
\begin{lem}\label{lem-c-ybe-with-blocks}
	Let $c: (\mathbb{C}^n )^{\otimes 2} \rightarrow (\mathbb{C}^n )^{\otimes 2} $  be a  $R$-matrix, with a partition as in Equation (\ref{eqn-c-blocks}). Then, for every $1 \leq i,k \leq n$, the following equation holds:
	\begin{equation}\label{eqn-c-ybe-condition}
c\,(B_{ik}\otimes \mat{I}_{n} )\,c\,=\,\sum\limits_{\ell=1}^{\ell=n}	(B_{i\ell}\otimes \mat{I}_{n} )	\,c\,(B_{\ell k}\otimes \mat{I}_{n} )
	\end{equation}
	\end{lem}
\begin{proof}
	We compute the square matrices 		$	(\mat{I}_{n} \otimes c)\,(c \otimes \mat{I}_{n} )\,(\mat{I}_{n} \otimes c)$
	and 
		$	(c \otimes \mat{I}_{n} )\,(\mat{I}_{n} \otimes c)\,	(c \otimes \mat{I}_{n} )$:

			\begin{equation}\label{eqn-c23c12c23}
			(\mat{I}_{n} \otimes c)\,(c \otimes \mat{I}_{n} )\,(\mat{I}_{n} \otimes c)=
		\begin{pmatrix}
		c(B_{11}\otimes \mat{I}_{n} )c
		
		& \rvline & ..
		& \rvline & c(B_{1k}\otimes \mat{I}_{n} )c&\rvline &.. &\rvline& 	c(B_{1n}\otimes \mat{I}_{n} )c\\
		\hline
		...
		
		& \rvline & ..
		& \rvline &..&\rvline &.. &\rvline& ... \\
		
		\hline
		c(B_{i1}\otimes \mat{I}_{n} )c
		
		& \rvline & 	..
		& \rvline &c(B_{ik}\otimes \mat{I}_{n} )c&\rvline &.. &\rvline& 	c(B_{in}\otimes \mat{I}_{n} )c\\
		\hline
		...
		
		& \rvline & ..
		& \rvline &..&\rvline &.. &\rvline& ... \\
		\hline
		
		c(B_{n1}\otimes \mat{I}_{n} )c
		
		& \rvline & 	..
		& \rvline &c(B_{nk}\otimes \mat{I}_{n} )c&\rvline &.. &\rvline& 	c(B_{nn}\otimes \mat{I}_{n} ) c\\
		\end{pmatrix}
			\end{equation}
				That is, the square block of size $n^2$ at position $ik$ in $(\mat{I}_{n} \otimes c)\,(c \otimes \mat{I}_{n} )\,(\mat{I}_{n} \otimes c)$ has the form 
			$c(B_{ik}\otimes \mat{I}_{n} )c$. The matrix $(c \otimes \mat{I}_{n} )\,(\mat{I}_{n} \otimes c)\,(c \otimes \mat{I}_{n} )$ is equal to:

			\begin{equation*}\label{eqn-c12c23c12}
			\begin{pmatrix}
			B_{11}\otimes \mat{I}_{n} 
			
			& \rvline & ..
			& \rvline & B_{1j}\otimes \mat{I}_{n} &\rvline &.. &\rvline& 	B_{1n}\otimes \mat{I}_{n} \\
				\hline
			...
			
			& \rvline & ..
			& \rvline &..&\rvline &.. &\rvline& ... \\
		
			\hline
			B_{i1}\otimes \mat{I}_{n} 
			
			& \rvline & 	..
			& \rvline &B_{ij}\otimes \mat{I}_{n} &\rvline &.. &\rvline& 	B_{in}\otimes \mat{I}_{n} \\
			\hline
			...
			
			& \rvline & ..
			& \rvline &..&\rvline &.. &\rvline& ... \\
			\hline
			
			B_{n1}\otimes \mat{I}_{n} 
			
			& \rvline & 	..
			& \rvline &B_{nj}\otimes \mat{I}_{n} &\rvline &.. &\rvline& 	B_{nn}\otimes \mat{I}_{n} )\\
				\end{pmatrix}\cdot
				\begin{pmatrix}
				c(B_{11}\otimes \mat{I}_{n} )
				
				& \rvline & ..
				& \rvline & c(B_{1k}\otimes \mat{I}_{n} )&\rvline &.. &\rvline& 	c(B_{1n}\otimes \mat{I}_{n} )\\
				\hline
				...
				
				& \rvline & ..
				& \rvline &..&\rvline &.. &\rvline& ... \\
				
				\hline
				c(B_{j1}\otimes \mat{I}_{n} )
				
				& \rvline & 	..
				& \rvline &c(B_{jk}\otimes \mat{I}_{n} )&\rvline &.. &\rvline& 	c(B_{jn}\otimes \mat{I}_{n} )\\
				\hline
				...
				
				& \rvline & ..
				& \rvline &..&\rvline &.. &\rvline& ... \\
				\hline
				
				c(B_{n1}\otimes \mat{I}_{n} )
				
				& \rvline & 	..
				& \rvline &c(B_{nk}\otimes \mat{I}_{n} )&\rvline &.. &\rvline& 	c(B_{nn}\otimes \mat{I}_{n} ) \\
				\end{pmatrix}
			\end{equation*}
		That is, the square block of size $n^2$ at position $ik$ in  $(	(c \otimes \mat{I}_{n} )\,(\mat{I}_{n} \otimes c)\,(c \otimes \mat{I}_{n} ))$  has the form 
		$(B_{i1}\otimes \mat{I}_{n} )	c(B_{1k}\otimes \mat{I}_{n} )\,+	(B_{i2}\otimes \mat{I}_{n} )	c(B_{2k}\otimes \mat{I}_{n} )+\,	...+(B_{in}\otimes \mat{I}_{n} )	c(B_{nk}\otimes \mat{I}_{n} )$, that is 
		$\sum\limits_{\ell=1}^{\ell=n}	(B_{i\ell}\otimes \mat{I}_{n} )	c(B_{\ell k}\otimes \mat{I}_{n} )$. As $c$ satisfies the YBE,  (\ref{eqn-c-ybe-condition}) holds, for every $1 \leq i,k \leq n$.
\end{proof}
\subsection{ Construction of new $R$-matrices with the Tracy-Singh product  of matrices}
Let $c: (\mathbb{C}^n )^{\otimes 2} \rightarrow (\mathbb{C}^n )^{\otimes 2} $ and $d:(\mathbb{C}^m )^{\otimes 2}\rightarrow (\mathbb{C}^m )^{\otimes 2}$ be   $R$-matrices.    Let $c \boxtimes d$ denote their  Tracy-Singh product,  where  $c$ and $d$ have  a   block partition as in Equation (\ref{eqn-c-blocks}). 	Let $B_{ij}$ and  $B'_{ij}$ denote the blocks of $c$ and $d$ respectively.  Then the matrix $c \boxtimes d$  looks as follows:
\begin{equation}\label{eqn-ctracyc-blocks}
c \boxtimes d=\,
\begin{pmatrix}
B_{11}\otimes  B'_{11}
 ..
&.. B_{11}\otimes B'_{1m}&\rvline &.. &\rvline& 	B_{1n}\otimes B'_{11}.. & ..	B_{1n}\otimes B'_{1m} \\
\hdashline[2pt/2pt]
...
& ..
.&\rvline &.. &\rvline& ...&... \\
\hdashline[2pt/2pt]
B_{11}\otimes B'_{m1}
	..
 &..B_{11}\otimes B'_{mm} &\rvline &.. &\rvline& B_{1n}\otimes B'_{m1} ..&	..B_{1n}\otimes B'_{mm} \\
\hline
...
& .
..&\rvline &.. &\rvline& ...&... \\
\hdashline[2pt/2pt]
B_{n1}\otimes B'_{m1}
	..
&..B_{n1}\otimes B'_{mm}&\rvline &.. &\rvline& B_{nn}\otimes B'_{m1}..&..	B_{nn}\otimes B'_{mm} \\
\end{pmatrix}
\end{equation}
We show that $c \boxtimes d$ is an $R$-matrix, which means  there exists a linear operator  $\tilde{c_d}:\mathbb{C}^{nm}\otimes\mathbb{C}^{nm}\rightarrow \mathbb{C}^{nm} \otimes \mathbb{C}^{nm}$ and  an ordered basis of  $(\mathbb{C}^{nm})^{\otimes 2}$,   such that  $\tilde{c_d}$ is represented by $c \boxtimes d$ with respect to this basis. For that, we show that the square matrix of size $(nm)^2$, $c \boxtimes d$ is invertible and   satisfies Equation (\ref{eqn-ybe-matrix}). Clearly, if $c$ and $d$ are  invertible, then from  Theorem \ref{thm-tracy} $(vii)$, $c \boxtimes d$ is also invertible and furthermore $(c \boxtimes d)^{-1}\,=c^{-1} \boxtimes d^{-1}$. In the rest of this section, we show that $c \boxtimes d$ satisfies  Equation (\ref{eqn-tilde-c-ybe-equation}):
\begin{equation}\label{eqn-tilde-c-ybe-equation}
\scalemath{0.8}{
((c \boxtimes d)\otimes \mat{I}_{nm})\,( \mat{I}_{nm}\otimes(c \boxtimes d))\,((c \boxtimes d)\otimes \mat{I}_{nm})\;=\;( \mat{I}_{nm}\otimes(c \boxtimes d))\,((c \boxtimes d)\otimes \mat{I}_{nm})\,(( \mat{I}_{nm}\otimes(c \boxtimes d))}
\end{equation}
We now turn to the understanding of Equation (\ref{eqn-tilde-c-ybe-equation}). In Equation (\ref{eqn-tilde-c23-blocks}), we describe the square matrices of size $(nm)^3$,  $\mat{I}_{nm}\otimes(c \boxtimes d)$ at left and $(c \boxtimes d)\otimes\mat{I}_{nm}$ at right:
	\begin{equation}\label{eqn-tilde-c23-blocks}
	\scalemath{0.75}{
\begin{pmatrix}
		c \boxtimes d		
		& \rvline &	\bigzero_{(nm)^2}
		& \rvline &... &\rvline & 	\bigzero_{(nm)^2}\\
		\hline
		\bigzero_{(nm)^2}		
		& \rvline & c \boxtimes d
		& \rvline &... &\rvline & 	\bigzero_{(nm)^2}\\
		\hline
		...		
		& \rvline & ...
		& \rvline &... &\rvline & ... \\
		\hline
		\bigzero_{(nm)^2}		
		& \rvline & 	\bigzero_{(nm)^2}
		& \rvline &... &\rvline & c \boxtimes d
		\end{pmatrix}
		\;\;\textrm{and}\;\;
	\begin{pmatrix}
		B_{11}\otimes B'_{11}\otimes \mat{I}_{nm} 
		& \rvline & B_{11}\otimes B'_{12}\otimes \mat{I}_{nm}
		& \rvline &... &\rvline & B_{1n} \otimes B'_{1m}\otimes \mat{I}_{nm}\\
		\hline
		B_{11}\otimes B'_{21}\otimes \mat{I}_{nm} 		
		& \rvline & B_{11}\otimes B'_{22}\otimes \mat{I}_{nm}
		& \rvline &... &\rvline & B_{1n} \otimes B'_{2m}\otimes \mat{I}_{nm}\\
		\hline
		...
		& \rvline & ...
		& \rvline &... &\rvline & ... \\
		\hline		
		B_{n1}\otimes B'_{m1}\otimes \mat{I}_{nm}		
		& \rvline & B_{n1}\otimes B'_{m2}\otimes \mat{I}_{nm}
		& \rvline &... &\rvline & B_{nn} \otimes B'_{mm}\otimes \mat{I}_{nm}
		\end{pmatrix}}
		\end{equation}

\begin{lem}\label{lem-tilde-c-condition-for YBE}
	Let $c: (\mathbb{C}^n )^{\otimes 2} \rightarrow (\mathbb{C}^n )^{\otimes 2} $ and $d:(\mathbb{C}^m )^{\otimes 2}\rightarrow (\mathbb{C}^m )^{\otimes 2}$ be   $R$-matrices.  Let $c \boxtimes d$ denote their  Tracy-Singh product,  where  $c$ and $d$ have  a   block partition as in (\ref{eqn-c-blocks}). Let $B_{ij}$ and  $B'_{ij}$ denote the blocks of $c$ and $d$ respectively.  Let  $\lceil ..\rceil$  denotes the ceil function  and  $\bar{i},\bar{k}$ denote the residue modulo $m$ of $i$ and $k$ respectively and whenever the residue is $0$ it is replaced by $m$.
	 Then $c \boxtimes d$  is an $R$-matrix if and only if  the following equality of square matrices of size $(nm)^2$ holds,  for every $1 \leq i,k\leq nm$:
	\begin{equation}\label{eqn-tilde-c-ybe-condition}
	\scalemath{0.75}{
	(c \boxtimes d)\,(B_{ \lceil \frac{i}{m}\rceil\,\lceil \frac{k}{m}\rceil}\,\otimes\, B'_{\bar{i}\,\bar{k}}\,\otimes \,\mat{I}_{nm})\,(c \boxtimes d)\,=\,\sum\limits_{\ell=1}^{\ell=nm}	\,	(B_{ \lceil \frac{i}{m}\rceil\,\lceil \frac{\ell}{m}\rceil}\,\otimes\, B'_{\bar{i}\,\bar{\ell}}\,\otimes \,\mat{I}_{nm})\,(c \boxtimes d)\,
	(B_{ \lceil \frac{\ell}{m}\rceil\,\lceil \frac{k}{m}\rceil}\,\otimes\, B'_{\bar{\ell}\,\bar{k}}\,\otimes \,\mat{I}_{nm})}
	\end{equation}
\end{lem}

\begin{proof}
	We compute the matrix of size $(nm)^3$, $	(\mat{I}_{nm} \otimes (c \boxtimes d))\,((c \boxtimes d)\otimes \mat{I}_{nm} )\,(\mat{I}_{nm} \otimes (c \boxtimes d))$:
	\begin{equation*}\label{eqn-tilde-c23c12c23}
\begin{pmatrix}
(c \boxtimes d)(B_{11}\otimes B'_{11}\otimes \mat{I}_{nm})(c \boxtimes d)

& \rvline &... &\rvline & (c \boxtimes d)(B_{1n} \otimes B'_{1m}\otimes \mat{I}_{nm})(c \boxtimes d)\\
\hline
(c \boxtimes d)(B_{11}\otimes B'_{21}\otimes \mat{I}_{nm})(c \boxtimes d)

& \rvline &... &\rvline & (c \boxtimes d)(B_{1n} \otimes B'_{2m}\otimes \mat{I}_{nm})(c \boxtimes d)\\
\hline
...

& \rvline & ...
& \rvline &... &\rvline  \\
\hline

(c \boxtimes d)(B_{n1}\otimes B'_{m1}\otimes \mat{I}_{nm})(c \boxtimes d)

& \rvline &... &\rvline & (c \boxtimes d)(B_{nn} \otimes B'_{mm}\otimes \mat{I}_{nm})(c \boxtimes d)\\

\end{pmatrix}
\end{equation*}
	That is, the square block of size $n^2m^2$ at position $ik$ in $	(\mat{I}_{nm} \otimes (c \boxtimes d))\,((c \boxtimes d)\otimes \mat{I}_{nm} )\,(\mat{I}_{nm} \otimes (c \boxtimes d))$ has the following form:
	\begin{equation}\label{eqn-block-ik-tilde-c23-c12-c23}
	(c \boxtimes d)\,(B_{ \lceil \frac{i}{m}\rceil\,\lceil \frac{k}{m}\rceil}\,\otimes\, B'_{\bar{i}\,\bar{k}}\,\otimes \,\mat{I}_{nm})\,(c \boxtimes d)
	\end{equation}
	We compute the square matrix of size $(nm)^3$, $((c \boxtimes d)\otimes \mat{I}_{nm})\,(\mat{I}_{nm} \otimes (c \boxtimes d))\,((c \boxtimes d)\otimes\mat{I}_{nm} )$:
\begin{equation*}\label{eqn-tilde-c12c23c12}
			\scalemath{0.8}{\begin{pmatrix}
		B_{11}\otimes B'_{11}\otimes \mat{I}_{nm}		
		& \rvline & ..
		&\rvline& 	B_{1n}\otimes B'_{1m}\otimes \mat{I}_{nm}\\
		\hline
		...
		& \rvline & ..
		&\rvline& ... \\		
		\hline
		B_{i1}\otimes B'_{11}\otimes \mat{I}_{nm}		
		& \rvline & 	..
		&\rvline& 	B_{in}\otimes B'_{1m}\otimes \mat{I}_{nm}\\
		\hline
		...	
		& \rvline & ..
		&\rvline& ... \\
		\hline
		B_{n1}\otimes B'_{m1}\otimes \mat{I}_{nm}		
		& \rvline & 	..
		&\rvline& 	B_{nn}\otimes B'_{mm}\otimes \mat{I}_{nm}\\
		\end{pmatrix}\cdot
		\begin{pmatrix}
		(c \boxtimes d)(B_{11}\otimes B'_{11}\otimes \mat{I}_{nm})		
		& \rvline & ..
		& \rvline& 		(c \boxtimes d)(B_{1n}\otimes B'_{1m}\otimes \mat{I}_{nm})\\
		\hline
		...		
		& \rvline & ..
		& \rvline &..\\
		\hline
		(c \boxtimes d)(B_{i1}\otimes B'_{11}\otimes \mat{I}_{nm})		
		& \rvline & 	..
		& \rvline & 		(c \boxtimes d)(B_{in}\otimes B'_{1m}\otimes \mat{I}_{nm})\\
		\hline
		...
		& \rvline & ..
		& \rvline &..\\
		\hline		
		(c \boxtimes d)(B_{n1}\otimes B'_{m1}\otimes \mat{I}_{nm})
		& \rvline & 	..
		& \rvline & 	(c \boxtimes d)(B_{nn}\otimes B'_{mm}\otimes \mat{I}_{nm}) \\
		\end{pmatrix} }
	\end{equation*}
 The square block of size $n^2m ^2$ at position $ik$ in $((c \boxtimes d)\otimes \mat{I}_{nm})\,(\mat{I}_{nm} \otimes (c \boxtimes d))\,((c \boxtimes d)\otimes\mat{I}_{nm} )$:
	\begin{equation}\label{eqn-block-ik-tilde-c12-c23-c12}
	\sum\limits_{\ell=1}^{\ell=nm}	\,	(B_{ \lceil \frac{i}{m}\rceil\,\lceil \frac{\ell}{m}\rceil}\,\otimes\, B'_{\bar{i}\,\bar{\ell}}\,\otimes \,\mat{I}_{nm})\,(c \boxtimes d)\,
	(B_{ \lceil \frac{\ell}{m}\rceil\,\lceil \frac{k}{m}\rceil}\,\otimes\, B'_{\bar{\ell}\,\bar{k}}\,\otimes \,\mat{I}_{nm})
	\end{equation}
	So, $c \boxtimes d$  is an $R$-matrix, that is $c\boxtimes d$ satisfies Equation (\ref{eqn-tilde-c-ybe-equation}) if and only if  Equation (\ref{eqn-tilde-c-ybe-condition}) holds,  for every $1 \leq i,k\leq nm$.
\end{proof}
The reader may wonder where do the $ \lceil \frac{i}{m}\rceil$, $\lceil \frac{k}{m}\rceil$, $\bar{i}$  and $\bar{k}$ come from in Lemma \ref{lem-tilde-c-condition-for YBE}. The reason is purely combinatorial. Indeed, in $c \boxtimes d$,  the  first $m$  horizontal and the first $m$ vertical blocks look like $B_{11}  \otimes B'_{pq}$, $1 \leq p,q \leq m$,  the first $m$ horizontal  and the next  $m$ vertical blocks look like $B_{12}  \otimes B'_{pq}$, $1 \leq p,q \leq m$,  and so on until  $B_{1n}  \otimes B'_{pq}$, $1 \leq p,q \leq m$.  This explains the apparition of  $ \lceil \frac{i}{m}\rceil$ and $\lceil \frac{k}{m}\rceil$.  To explain the apparition of $\bar{i}$ and $\bar{k}$, one needs to make the same process with focus on $B'_{pq}$, $1 \leq p,q \leq m$. In size to prove that $c \boxtimes d$  is an $R$-matrix,  we will prove that it satisfies the conditions of Lemma \ref{lem-tilde-c-condition-for YBE}. The following two lemmas are very useful and important for the proof. \begin{lem}\label{lem-A-I-tracy-B-I}
	Let $A$ and $B$ be square matrices of size $n$ and $m$ respectively. Then, with a certain block partition in the Tracy-Singh product, the following equation holds:
	\begin{equation}\label{eqn-A-I-tracy-B-I}
	(A\otimes \mat{I}_{n})\,\boxtimes\,(B\otimes \mat{I}_{m})\,=\,A\otimes B \otimes  \mat{I}_{nm}
	\end{equation}
\end{lem}
\begin{proof}
	The matrix $A\otimes B \otimes  \mat{I}_{nm}$ is a square matrix of size $(nm)^2$, which can be partitioned into  $(nm)^2$  square blocks of size $nm$, where each block is a scalar matrix, described as:
	\begin{equation}\label{eqn-A-B-tensor In2}
	A\otimes B \otimes  \mat{I}_{nm}=
	\begin{pmatrix}
	a_{11}b_{11}\mat{I}_{nm}

	& \rvline & ..
	& \rvline & a_{11}b_{1m}\mat{I}_{nm}&\rvline &.. &\rvline& 	a_{1n}b_{1m}\mat{I}_{nm}\\
	\hline
	...
	
	& \rvline & ..
	& \rvline &..&\rvline &.. &\rvline& ... \\
	
	\hline
	a_{i1}b_{11}\mat{I}_{nm}
	
	& \rvline & 	..
	& \rvline &a_{i1}b_{1m}\mat{I}_{nm}&\rvline &.. &\rvline& 	a_{in}b_{mm}\mat{I}_{nm}\\
	\hline
	...
	
	& \rvline & ..
	& \rvline &..&\rvline &.. &\rvline& ... \\
	\hline
	
	a_{n1}b_{m1}\mat{I}_{nm}
	
	& \rvline & 	..
	& \rvline &a_{n1}b_{mm}\mat{I}_{nm}&\rvline &.. &\rvline& 	a_{nn}b_{mm}\mat{I}_{nm}\\
	\end{pmatrix}
	\end{equation}
	\begin{equation}\label{eqn-A-In-tracy-B_In}
	\scalemath{0.85}{
	(A\otimes \mat{I}_{n})\,\boxtimes\,(B\otimes \mat{I}_{m})\,=
	\begin{pmatrix}
	a_{11}\mat{I}_{n}
	
	& \rvline & 	a_{12}\mat{I}_{n}
	& \rvline &... &\rvline & 	a_{1n}\mat{I}_{n}\\
	\hline
	a_{21}\mat{I}_{n}
	
	& \rvline & 	a_{22}\mat{I}_{n}
	& \rvline &... &\rvline & 	a_{2n}\mat{I}_{n}\\
	\hline
	...
	
	& \rvline & ...
	& \rvline &... &\rvline & ... \\
	\hline
	
	a_{n1}\mat{I}_{n}
	
	& \rvline & 	a_{n2}\mat{I}_{n}
	& \rvline &... &\rvline & 	a_{nn}\mat{I}_{n}\\
	
	\end{pmatrix}\boxtimes
	\begin{pmatrix}
	b_{11}\mat{I}_{m}
	
	& \rvline & 	b_{12}\mat{I}_{m}
	& \rvline &... &\rvline & 	b_{1m}\mat{I}_{m}\\
	\hline
	b_{21}\mat{I}_{m}
	
	& \rvline & 	b_{22}\mat{I}_{m}
	& \rvline &... &\rvline & 	b_{2m}\mat{I}_{m}\\
	\hline
	...
	
	& \rvline & ...
	& \rvline &... &\rvline & ... \\
	\hline
	
	b_{m1}\mat{I}_{m}
	
	& \rvline & 	b_{m2}\mat{I}_{m}
	& \rvline &... &\rvline & 	b_{mm}\mat{I}_{m}\\
	\end{pmatrix}	}
	\end{equation}
	From the definition of the Tracy-Singh product with the block partition described above and since $\mat{I}_{n} \otimes \mat{I}_{m}\,=\,\mat{I}_{nm}$, we have 
	$(A\otimes \mat{I}_{n})\,\boxtimes\,(B\otimes \mat{I}_{m})\,=\,A\otimes B \otimes  \mat{I}_{nm}$.
\end{proof}
\begin{prop}\label{lem-box-same}
	Let $c: (\mathbb{C}^n )^{\otimes 2} \rightarrow (\mathbb{C}^n )^{\otimes 2} $ and $d:(\mathbb{C}^m )^{\otimes 2}\rightarrow (\mathbb{C}^m )^{\otimes 2}$ be   $R$-matrices.  Let $c \boxtimes d$ denote their  Tracy-Singh product,  where  $c$ and $d$ have  a   block partition as in (\ref{eqn-c-blocks}). Let $B_{ij}$ and  $B'_{ij}$ denote the blocks of $c$ and $d$ respectively.    $\forall\, 1 \leq i,k\leq n,\,\forall \, 1 \leq p,q \leq m$:
	\begin{equation}\label{eqn-expression-box-itself}
	\scalemath{0.86}{
	(c \boxtimes d)\,(B_{ ik}\,\otimes\, B'_{pq}\,\otimes \,\mat{I}_{nm})\,(c \boxtimes d)\,=\,\sum\limits_{u=1}^{u=n}\,\sum\limits_{v=1}^{v=m}	\,	(B_{ iu}\,\otimes\, B'_{pv}\,\otimes \,\mat{I}_{nm})\,(c \boxtimes d)\,
	(B_{ u k}\,\otimes\, B'_{vq}\,\otimes \,\mat{I}_{nm})}
	\end{equation}
\end{prop}
\begin{proof}
	From Lemma \ref{lem-c-ybe-with-blocks}, 	for every $1 \leq i,k \leq n$, $1 \leq p,q \leq m$:
	\begin{gather*}
	c(B_{ik}\otimes \mat{I}_{n} )c\,=\,\sum\limits_{u=1}^{u=n}	(B_{iu}\otimes \mat{I}_{n} )	c(B_{u k}\otimes \mat{I}_{n} )\\
	d(B'_{pq}\otimes \mat{I}_{m} )d\,=\,\sum\limits_{v=1}^{v=m}	(B'_{pv}\otimes \mat{I}_{m} )	d(B'_{v q}\otimes \mat{I}_{m} )
	\end{gather*}
	So, $(c(B_{ik}\otimes \mat{I}_{n} )c)\,\boxtimes( d(B'_{pq}\otimes \mat{I}_{m} )d)\,=\,(\sum\limits_{u=1}^{u=n}	(B_{iu}\otimes \mat{I}_{n} )	c(B_{u k}\otimes \mat{I}_{n} ))\,\boxtimes \,(\sum\limits_{v=1}^{v=m}	(B'_{pv}\otimes \mat{I}_{m} )	d(B'_{v q}\otimes \mat{I}_{m} ))$. \\
	We compute  	$(c(B_{ik}\otimes \mat{I}_{n} )c)\,\boxtimes( d(B'_{pq}\otimes \mat{I}_{m} )d)$, using   Theorem  \ref{thm-tracy} $(v)$ and then  Lemma \ref{lem-A-I-tracy-B-I}:
	\begin{gather*}
	(c(B_{ik}\otimes \mat{I}_{n} )c)\,\boxtimes( d(B'_{pq}\otimes \mat{I}_{m} )d)\,=\\
	(c\boxtimes d)\,((B_{ik}\otimes \mat{I}_{n} )\,\boxtimes\, (B'_{pq}\otimes \mat{I}_{m} ))\,(c \boxtimes d)\,=\\
	(c\boxtimes d)\, (B_{ik}\otimes B'_{pq}\otimes \mat{I}_{nm})\,\,(c\boxtimes d)
	\end{gather*}
	We compute $	(\sum\limits_{u=1}^{u=n}	(B_{iu}\otimes \mat{I}_{n} )\,	c\,(B_{u k}\otimes \mat{I}_{n} ))\,\boxtimes \,(\sum\limits_{v=1}^{v=m}	(B'_{pv}\otimes \mat{I}_{m} )\,	d\,(B'_{v q}\otimes \mat{I}_{m} ))$,  using first  Theorem \ref{thm-tracy} $(iv)$ and $(v)$ and next   Lemma \ref{lem-A-I-tracy-B-I}:
	\begin{gather*}
(\sum\limits_{u=1}^{u=n}	(B_{iu}\otimes \mat{I}_{n} )\,	c\,(B_{u k}\otimes \mat{I}_{n} ))\,\boxtimes \,(\sum\limits_{v=1}^{v=m}	(B'_{pv}\otimes \mat{I}_{m} )\,	d\,(B'_{v q}\otimes \mat{I}_{m} ))=\\
	\sum\limits_{u=1}^{u=n}\,\sum\limits_{v=1}^{v=m} \, 
	((B_{iu}\otimes \mat{I}_{n} )	c\,(B_{u k}\otimes \mat{I}_{n} ))\, \boxtimes\,((B'_{pv}\otimes \mat{I}_{m} )	\,d\,(B'_{v q}\otimes \mat{I}_{m} ))\,=\\
	\sum\limits_{u=1}^{u=n}\,\sum\limits_{v=1}^{v=m} \, 
	((B_{iu}\otimes \mat{I}_{n})\boxtimes\,(B'_{pv}\otimes \mat{I}_{m} ))\,\,(c \boxtimes d)\,\,	((B_{u k}\otimes \mat{I}_{n} )\,  \boxtimes (B'_{v q}\otimes \mat{I}_{m} ))\,=\\
	\sum\limits_{u=1}^{u=n}\,\sum\limits_{v=1}^{v=m} \, 
	(B_{iu}\otimes \,B'_{pv}\otimes \mat{I}_{nm} )\,\,(c \boxtimes d)\,\,	(B_{u k}\otimes \, B'_{v q}\otimes \mat{I}_{nm} )
	\end{gather*}
	So, Equation (\ref{eqn-expression-box-itself}) holds.
	\end{proof}
We are now able to formulate  Theorem \ref{thm1} more precisely  and give its proof.
\begin{thm}\label{theo-tracy-unitary}
		Let $c: (\mathbb{C}^n )^{\otimes 2} \rightarrow (\mathbb{C}^n )^{\otimes 2} $ and $d:(\mathbb{C}^m )^{\otimes 2}\rightarrow (\mathbb{C}^m )^{\otimes 2}$ be   $R$-matrices.  Let $c \boxtimes d$ their  Tracy-Singh product,  where  $c$ and $d$ have  a   block partition as in (\ref{eqn-c-blocks}). 	 	Then
	\begin{enumerate}[(i)]
			\item    $c\boxtimes d:\,(\mathbb{C}^{nm} )^{\otimes 2} \rightarrow (\mathbb{C}^{nm} )^{\otimes 2}$  is also an $R$-matrix.
			\item 	If $c$ and $d$ are unitary, then $c \boxtimes d$ is also unitary.		
		\end{enumerate}
		Moreover, for every natural numbers $k, \ell \geq 1$,    there exists a  solution of the   Yang-Baxter equation,  $\tilde{C}: \mathcal{V}\otimes \mathcal{V}\rightarrow  \mathcal{V} \otimes \mathcal{V}$, with $\operatorname{dim}(\mathcal{V})=n^{k}m^{\ell}$, induced from $c$ and $d$.  
	\end{thm}
\begin{proof}
		 We show that $c\boxtimes d$ satisfies the condition in Lemma \ref{lem-tilde-c-condition-for YBE}. Let $B_{ik}$, and  $B'_{pq}$ denote the blocks of $c$ and $d$ respectively.  
		From Proposition  \ref{lem-box-same},  for every $1 \leq i,k \leq n$,  $1 \leq ,p,q \leq m$:\\
	$(c \boxtimes d)\,(B_{ ik}\,\otimes\, B'_{pq}\,\otimes \,\mat{I}_{nm})\,(c \boxtimes d)\,=\,\sum\limits_{u=1}^{u=n}\,\sum\limits_{v=1}^{v=m}	\,	(B_{ iu}\,\otimes\, B'_{pv}\,\otimes \,\mat{I}_{nm})\,(c \boxtimes d)\,
	(B_{ u k}\,\otimes\, B'_{vq}\,\otimes \,\mat{I}_{nm})$.\\
	So, in particular, if we replace $i$ by $\lceil \frac{i}{m}\rceil$, $k$ by $\lceil \frac{k}{m}\rceil$, $p$ by $\bar{i}$ and $q$ by $\bar{k}$ in this equation:
\begin{gather*}
		(c \boxtimes d)\,(B_{ \lceil \frac{i}{m}\rceil\,\lceil \frac{k}{m}\rceil}\,\otimes\, B'_{\bar{i}\,\bar{k}}\,\otimes \,\mat{I}_{nm})\,(c \boxtimes d)\,=\\
		\sum\limits_{u=1}^{u=n}\,\sum\limits_{v=1}^{v=m}	\,	\,	(B_{ \lceil \frac{i}{m}\rceil\,u}\,\otimes\, B'_{\bar{i}\,v}\,\otimes \,\mat{I}_{nm})\,(c \boxtimes d)\,
		(B_{ u\,\lceil \frac{k}{m}\rceil}\,\otimes\, B'_{v\,\bar{k}}\,\otimes \,\mat{I}_{nm})
		\end{gather*}
	From  Lemma \ref{lem-tilde-c-condition-for YBE}, it remains to show that:
	\begin{gather}
	\sum\limits_{u=1}^{u=n}\,\sum\limits_{v=1}^{v=m}	\,	\,	(B_{ \lceil \frac{i}{m}\rceil\,u}\,\otimes\, B'_{\bar{i}\,v}\,\otimes \,\mat{I}_{nm})\,(c \boxtimes d)\,
	(B_{ u\,\lceil \frac{k}{m}\rceil}\,\otimes\, B'_{v\,\bar{k}}\,\otimes \,\mat{I}_{nm})\,= \nonumber\\
	\sum\limits_{\ell=1}^{\ell=nm}	\,	(B_{ \lceil \frac{i}{m}\rceil\,\lceil \frac{\ell}{m}\rceil}\,\otimes\, B'_{\bar{i}\,\bar{\ell}}\,\otimes \,\mat{I}_{nm})\,(c \boxtimes d)\,
	(B_{ \lceil \frac{\ell}{m}\rceil\,\lceil \frac{k}{m}\rceil}\,\otimes\, B'_{\bar{\ell}\,\bar{k}}\,\otimes \,\mat{I}_{nm})
	\label{eqn-chge-indices}
	\end{gather}
	There are exactly $m$ values of $\ell$, $1 \leq \ell \leq nm$,  for which $\lceil \frac{\ell}{m}\rceil $ is equal to a given $u$, with $1 \leq u\leq n$. Indeed, for every $\ell \in \{(u-1)m+1\,,\, (u-1)m+2, ...,\,(u)m  \}$, then $\lceil \frac{\ell}{m}\rceil \,=\,u$. Moreover, for each of these $m$  values of $\ell$, it holds that $\bar{\ell}$ is different, that is $\bar{\ell}$ achieves all the values $1,..,m$. So, we can replace in the equation above $\lceil \frac{\ell}{m}\rceil $ by  $u$ and $\bar{\ell}$ by $v$, with the  indices $u$ and $v$ running from $1$ to $n$ and $m$ respectively.   That is, Equation (\ref{eqn-chge-indices}) holds, which proves that $c \boxtimes d$ is an $R$-matrix.\\
 $(ii)$ We show that if $c$ and $d$ are unitary, then $(c \boxtimes d)(c \boxtimes d)^*\,=\,I_{n^2m^2}$. From  theorem \ref{thm-tracy}$(viii)$ and the properties of the complex conjugation,  $(c \boxtimes d)(c \boxtimes d)^*\,=  (c \boxtimes d)(c^* \boxtimes d^*)$.  From  theorem \ref{thm-tracy}$(v)$,  $(c \boxtimes d)(c^* \boxtimes d^*)= c c^*\boxtimes d d^* = I_{n^2} \boxtimes I_{m^2}=I_{n^2m^2}$, since the products $cc^*$ and $dd^*$ are defined and equal to  $I_{n^2}$ and  $I_{m^2}$ respectively.\\
  From Theorem \ref{thm-tracy} $(iii)$, the Tracy-Singh product is associative, so one can apply for example
 	$\underbrace{c \boxtimes c\boxtimes...c}_{k\, \textrm{times}}\,\boxtimes	\underbrace{d\boxtimes d\boxtimes...d}_{\ell\, \textrm{times}}$ and construct a solution  of the   Yang-Baxter equation,  $\tilde{C}: \mathcal{V}\otimes \mathcal{V}\rightarrow  \mathcal{V} \otimes \mathcal{V}$, with $\operatorname{dim}(\mathcal{V})=n^{k}m^{\ell}$, induced from $c$, $d$. 
\end{proof}
\hspace*{3mm} A question that arises   is what happens if we change the block partition of the matrices in  the Tracy-Singh product. The block partition in (\ref{eqn-c-blocks}) is very symmetric in the sense that all the blocks are square submatrices, and this is  the unique  one,  so in any other block partition there is no such symmetry.  From the test of several examples, it holds that the Tracy-Singh product of two $R$-matrices with a block partition different from the block partition in (\ref{eqn-c-blocks}) is not necessarily a  $R$-matrix. As an example, for $c$,   the  $R$-matrix from Example \ref{ex-permut-sol-2},  with the following partition:
	$c=	\begin{pmatrix}
	0 &	 \rvline & 0& 0 & \rvline &1\\
	0 &	 \rvline & 1& 0  	& \rvline &0\\
	\hline
	0 &	 \rvline & 0& 1 & \rvline &0\\
	1 &	 \rvline & 0& 0 	& \rvline &0\\
	\end{pmatrix}$.  If we denote $(c\boxtimes c)\otimes \mat{I}_{4}$ by $\hat{c}^{12}$  and $\mat{I}_{4}\otimes (c\boxtimes c)$ by $\hat{c}^{23}$,  with this block partition, then $c\boxtimes c$ is not an $R$-matrix.
	Indeed, $\hat{c}^{12}\hat{c}^{23}\hat{c}^{12} \,\neq \,\hat{c}^{23}\hat{c}^{12}\hat{c}^{23}$, since 
	the first row in $\hat{c}^{12}\hat{c}^{23}\hat{c}^{12}$ is  a zero row, while  $(\hat{c}^{23}\hat{c}^{12}\hat{c}^{23})_{1,18}\,=\,(\hat{c}^{23}\hat{c}^{12}\hat{c}^{23})_{1,26}\,=\,1$. 
\section{Connections with other constructions}
The Kronecker product  (or tensor product) of matrices is a fundamental concept in linear algebra and   the Tracy-Singh product of matrices is a  generalisation of it,   called sometimes the block Kronecker product, as  they  share many properties.  While the Kronecker product of two matrices  has a very natural interpretation, indeed  it represents the tensor product of the corresponding linear transformations, the Tracy-Singh product  is not much understood and it is not known whether it has a  general  interpretation.  However, in the special case that $A$ and $B$  of size $n^2\times p^2$ and $m^2\times q^2$ respectively,  with a partition into  blocks of the same size ($n\times p$ and $m\times q$ respectively),   we call  a  \emph{canonical partition},  
then $A \boxtimes B$ is the representing matrix of some operator (as described below).  Furthermore, if $A$ and $B$ are square matrices, then  $A \boxtimes B$ and $A \otimes B$ are similar matrices and the conjugating matrix is a permutation matrix, as described in Proposition \ref{prop-prop-box-cd}, which is a direct application of Theorem \ref{thm-tracy-K}.  
\begin{prop}\label{prop-prop-box-cd}
	Let $c$ and $d$ of size $n^2\times p^2$ and $m^2\times q^2$ respectively,  with a partition into  blocks of the same size ($n\times p$ and $m\times q$ respectively). 	Then		
	\begin{equation}\label{eqn-formula-ts-tensor-general}
	c \boxtimes d\,=\, 	(\mat{I}_{n}\,\otimes K_{mn} \,\otimes \mat{I}_{m})\,	\cdot\, (c\otimes d)\, \cdot\, (\mat{I}_{p}\,\otimes K_{pq}\,\otimes  \mat{I}_{q})
	\end{equation}
	If  $c$ and $d$  are square matrices, that is $n= p$ and $m=q$, then 
	\begin{enumerate}[(i)]
		\item $d \otimes c\,=\, K_{m^2n^2}\,\cdot\,(c \otimes d)\,\cdot\,K_{n^2m^2}$, where $K_{m^2n^2}\,=\,(K_{n^2m^2})^{-1}$.
		\item 	
		$c \boxtimes d\,=\, 	(\mat{I}_{n}\,\otimes K_{mn} \,\otimes \mat{I}_{m})\,	\cdot\, (c\otimes d)\, \cdot\, (\mat{I}_{n}\,\otimes K_{nm}\,\otimes  \mat{I}_{m})$
		
		\item 
		$d\boxtimes c\,=\, 	P \cdot\, (c\boxtimes d)\,\cdot\,P^{-1}$, where 
		$P=\, (\mat{I}_{m}\,\otimes K_{nm} \,\otimes \mat{I}_{n})\,	\cdot\, K_{m^2n^2}\,\cdot\, (\mat{I}_{n}\,\otimes K_{nm} \,\otimes \mat{I}_{m})$, where $ (\mat{I}_{n}\,\otimes K_{mn} \,\otimes \mat{I}_{m})^{-1}\,=\, (\mat{I}_{n}\,\otimes K_{nm}\,\otimes  \mat{I}_{m})$.
	\end{enumerate}
\end{prop}
In \cite{gandal1,gandal2}, the authors describe several algebraic operations on the set of $R$-matrices. They name one of these operations   \emph{the tensor product of $R$-matrices}, although it differs from the actual tensor product $\otimes$, and they  denote it by $\boxtimes$. For  $R$-matrices $c: (\mathbb{C}^n )^{\otimes 2} \rightarrow (\mathbb{C}^n )^{\otimes 2} $ and $d:(\mathbb{C}^m )^{\otimes 2}\rightarrow (\mathbb{C}^m )^{\otimes 2}$, their \emph{tensor product}    $c \boxtimes d\, \in \operatorname{End}(\mathbb{C}^n \otimes \mathbb{C}^m \otimes \mathbb{C}^n \otimes \mathbb{C}^m )$:
\begin{equation}\label{eqn-defn-box-gandal}
c \boxtimes d\,=\, F_{23}\,(c \otimes d)\,F_{23}
\end{equation}
where $F_{23}:\,\mathbb{C}^n \otimes \mathbb{C}^m \otimes \mathbb{C}^n \otimes \mathbb{C}^m \,\rightarrow \,\mathbb{C}^n \otimes \mathbb{C}^n \otimes \mathbb{C}^m \otimes \mathbb{C}^m$ is the flip unitary exchanging the two middle factors \cite[Section 4]{gandal1}.  A technical computation, suggested by G.  Lechner,  shows that 
the matrix $ \mat{I}_{n}\,\otimes K_{mn} \,\otimes \mat{I}_{m}$ in Proposition \ref{prop-prop-box-cd} $(ii)$ represents the  flip map $F_{23}$. That is,   in the special case that  $A$ and $B$ are square matrices  of size $n^2$ and $m^2$ respectively, with the canonical partition into blocks,  their  Tracy-Singh product represents the operator $c \boxtimes d\,=\, F_{23}\,(c \otimes d)\,F_{23}$, where $A$ and $B$ represent the operators $c: (\mathbb{C}^n )^{\otimes 2} \rightarrow (\mathbb{C}^n )^{\otimes 2} $ and $d:(\mathbb{C}^m )^{\otimes 2}\rightarrow (\mathbb{C}^m )^{\otimes 2}$,  respectively. Moreover, if  $c$ and $d$ are $R$-matrices,  $c \boxtimes d$ is also an $R$-matrix in both contexts. In \cite{majid-book}, S. Majid refers to the cabling operation  of $R$-matrices.
\begin{rem}\label{rem-ts-general-matrices}
	More generally,  if $A$ and $B$  of size $n^2\times p^2$ and $m^2\times q^2$ respectively,  with a canonical partition,   then from Equation (\ref{eqn-formula-ts-tensor-general}), one can show (a  technical computation) that  their  Tracy-Singh product represents also a linear transformation  $c\boxtimes d:\,\mathbb{C}^n \otimes \mathbb{C}^m \otimes \mathbb{C}^n \otimes \mathbb{C}^m\,\rightarrow\mathbb{C}^p \otimes \mathbb{C}^q \otimes \mathbb{C}^p\otimes \mathbb{C}^q$ of the form $F_{23}\,(c \otimes d)\,F_{23}$, where $c: (\mathbb{C}^n )^{\otimes 2} \rightarrow (\mathbb{C}^p )^{\otimes 2} $ and $d:(\mathbb{C}^m )^{\otimes 2}\rightarrow (\mathbb{C}^q )^{\otimes 2}$, and $F_{23}$  exchanges the two middle factors. 
\end{rem}
The Tracy-Singh product of matrices  has  also  a surprising connection with a categorical construction that we describe in the following. Indeed, it can be defined as the monoidal product (or a tensor  functor) in a particular category of vector spaces, in which the canonical partition into blocks is ensured. Before we get into details,  we give some definitions and refer to the books \cite{book-turaev}, \cite{book-tensor}, \cite{etingof-mit},  \cite{majid-book} and to \cite{deligne}, \cite{lopez} for more details. 

An abelian category $\mathcal{C}$ is  \emph{$\mathbb{C}$-linear}
if for  every  $X,Y \in \mathcal{C}$,  $\operatorname{Hom }_{\mathcal{C}}(X,Y)$  is a $\mathbb{C}$-vector space, and composition of morphisms is bilinear. A $\mathbb{C}$-linear abelian category $\mathcal{C}$ is \emph{locally finite} if  for every  $X,Y \in \mathcal{C}$, the vector space $\operatorname{Hom }_{\mathcal{C}}(X,Y)$
is finite dimensional over $\mathbb{C}$ and there exists a filtration $\{0\}=X_0\subset X_1\subset ...\subset X_{n-1}\subset X_n\,=X$
such that for every $i$, $X_i/X_{i-1}$  has only $\{0\}$ and itself as subojects \cite[Sections 1.3,1.11, 4.6]{book-tensor}. A locally finite $\mathbb{C}$-linear abelian rigid monoidal category $\mathcal{C}$ is a \emph{multitensor category over $\mathbb{C}$},  if the bifunctor $\otimes:  \mathcal{C} \times \mathcal{C} \rightarrow \mathcal{C}$ is bilinear on morphisms. 

In \cite{deligne}, P. Deligne defines a product of categories, denoted by $\boxtimes$ (also), that is called nowadays the \emph{Deligne's tensor product}. We denote it by $\boxtimes_P$ to differentiate it from the others.  Deligne's tensor product is not always defined, but if  $\mathcal{C}$ and $\mathcal{D}$ are  locally finite $\mathbb{C}$-linear abelian categories or multitensor categories, 
then  the Deligne's tensor  product of  $\mathcal{C}$ and $\mathcal{D}$ exists and  is also  a locally finite $\mathbb{C}$-linear abelian category or a multitensor category \cite{deligne,book-tensor,lopez}. Whenever
$\mathcal{C}$ and $\mathcal{D}$ are  multitensor categories, the  category $\mathcal{C} \boxtimes_P \mathcal{D}$ is monoidal with tensor product defined,   for every $C_1,C_2 \in \operatorname{Ob}(\mathcal{C})$ and  $D_1,D_2 \in \operatorname{Ob}(\mathcal{D})$:
\begin{gather*}
\hat{\otimes}:\,\mathcal{C} \boxtimes_P \mathcal{D} \times \mathcal{C} \boxtimes_P \mathcal{D}  \rightarrow \mathcal{C}\boxtimes_P\mathcal{D}\\
(C_1\boxtimes_P D_1)\hat{\otimes} (C_2\boxtimes_P D_2)\,=\,(C_1\,\otimes_{\mathcal{C}}\,C_2) \, \boxtimes_P     \,(D_1\,\otimes_{\mathcal{D}}\,D_2)
\end{gather*}
with  $\operatorname{Hom }_{\mathcal{C}\boxtimes_P\mathcal{D}}(C_1\boxtimes_P D_1,C_2\boxtimes_P D_2)\, \simeq \, \operatorname{Hom }_{\mathcal{C}}(C_1,C_2)\,\otimes \,\operatorname{Hom }_{\mathcal{D}}(D_1,D_2)$ \cite[Section 4.6]{book-tensor}.

The category $\operatorname{Vec}$ is a symmetric monoidal category, with   objects   all the finite dimensional vector spaces over a fixed field, let's say $\mathbb{C}$, and  morphisms the  linear transformations between vector spaces. The monoidal  product is the functor $\otimes:\, \operatorname{Vec}\times \operatorname{Vec}\,\rightarrow \, \operatorname{Vec}$ that sends  each pair of objects $(U,V)$ to $U\otimes V$ and each pair of morphisms $(f,g)$ to $f \otimes g$, and its unit object is  $\mathbb{C}$. 
That $\otimes$ is a functor means that $Id_X \otimes  Id_Y = Id_{X\otimes Y}$,  for all $X,Y \in \operatorname{Vec}$,  and 
$(g \otimes g')(f \otimes  f ') \,= \,gf \otimes g'f' $, for all pairs of composable morphisms $g,f$ and $g',f'$ in $\operatorname{Vec}$, properties satisfied by the tensor product. The braiding is the flip map $\tau = \,\{\tau_{X,Y} :\, X \otimes  Y\,\rightarrow\,Y \otimes X \}_{X,Y \in \operatorname{Vec}}$, which is a symmetry. 

Moreover,  $\operatorname{Vec}$  is  a multitensor category, and in fact it is a fusion category which is an even stronger property \cite[Lecture 3, p.32]{etingof-mit}. As such,  Deligne's tensor  product of  $\operatorname{Vec}$  with itself, $\operatorname{Vec}\,\boxtimes_P\,\operatorname{Vec}$ exists and  is also  a  multitensor category \cite[Section 4.6]{book-tensor}. In this particular case,  $\operatorname{Vec}\,\boxtimes_P \operatorname{Vec}$ is in fact $\operatorname{Vec}\otimes\operatorname{Vec}$   \cite[Lecture 9, p.90]{etingof-mit}, an additive category satisfying   $\operatorname{Hom }_{\operatorname{Vec}\,\otimes \operatorname{Vec}}(U\otimes V,\,U'\otimes V')\, =\, \operatorname{Hom }_{\operatorname{Vec}}(U,U')\,\otimes \,\operatorname{Hom }_{\operatorname{Vec}}(V,V') $  \cite[Section 1.11]{book-tensor}. Additionally, it is monoidal with tensor product defined,   for every $U_1,V_1,U_2,V_2 \, \in \operatorname{Ob}(\operatorname{Vec})$ and for every linear transformation $f_i:\,U_i\rightarrow U'_{i}$, $g_i \,V_i\rightarrow V'_{i}$, $i=1,2$,   as in  (\ref{eqn-defn-hat-tensor}):
\begin{gather}\label{eqn-defn-hat-tensor}
\hat{\otimes}:\, \operatorname{Vec}\otimes\operatorname{Vec} \,\times \operatorname{Vec}\otimes\operatorname{Vec}\,\rightarrow \operatorname{Vec}\otimes\operatorname{Vec}\\
((U_1\otimes V_1),(U_2\otimes V_2)) \,\mapsto\,(U_1\,\otimes U_2) \, \otimes     \,(V_1\,\otimes\,V_2) \nonumber\\
((f_1\otimes g_1),(f_2\otimes g_2)) \,\mapsto\,(f_1\,\otimes f_2) \, \otimes     \,(g_1\,\otimes\,g_2) \nonumber
\end{gather}
This defines a structure of  monoidal category on the subcategory of  $\operatorname{Vec}\otimes\operatorname{Vec}$  consisting of  $\otimes$-decomposable objects of the form $U \otimes V$, which extends  to the entire category $\operatorname{Vec}\otimes\operatorname{Vec}$  \cite[Lecture 9, p.90]{etingof-mit}.  The category $\operatorname{Vec}\otimes\operatorname{Vec}$ is symmetric with braiding  the flip map inherited from $\operatorname{Vec}$, indeed 
$(U_1\otimes V_1)\otimes (U_2\otimes V_2)\,\simeq\,(U_2\,\otimes V_2) \, \otimes     \,(U_1\,\otimes\,V_1)$, for every $U_1,V_1,U_2,V_2 \, \in \operatorname{Vec}$.  We want to concentrate on a particular subcategory of  $\operatorname{Vec}\otimes\operatorname{Vec}$:
\[\mathcal{Diag}\,=\,\{U\otimes U \,\mid \, U\in \operatorname{Vec}\}\subset\,\operatorname{Vec}\otimes\operatorname{Vec}\]
As  $\operatorname{Hom }_{\mathcal{Diag}}\,(U\otimes U,V\otimes V)\, =\, \operatorname{Hom }_{\operatorname{Vec}\otimes\operatorname{Vec}}\,(U\otimes U,V\otimes V)$,  for every $U, V\in \operatorname{Vec}$, $\mathcal{Diag}$ is a full subcategory. Moreover, it  is  (symmetric) monoidal, since it   is closed under the  tensor  $\hat{\otimes}\mid_{\mathcal{Diag}}$ and   the unit object $\mathbb{C}\otimes\mathbb{C}$ belongs to  it. So,  $\mathcal{Diag}\,\times \, \mathcal{Diag}$ is also a   category, as it is the cartesian product of categories and it has a unit object  $(\mathbb{C}\otimes\mathbb{C},\mathbb{C}\otimes\mathbb{C})$.  It is (symmetric) monoidal with tensor the application of $\hat{\otimes}$ on each coordinate.

We are ready now to present the connection between categories and the Tracy-Singh product of matrices. As we recall,   to apply the Tracy-Singh product on matrices they have to be partitioned into blocks. If the matrices have arbitrary sizes, they cannot be   necessarily partitioned into blocks of the same size. However, if a matrix  $A$ has size of the form $n^ 2\times p^2$,  $n,p$ not necessarily different, then there is a  canonical partition of  $c$, where all the  blocks have  the same size $n\times p$.   The existence of such a canonical partition is  ensured for representing matrices of morphisms  in $\mathcal{Diag}$, but not  in $\operatorname{Vec}\otimes\operatorname{Vec}$ in general.  We show  that
the Tracy-Singh product of the representing matrices  (with respect to standard bases for example) of morphisms  $c$ and $d$ in the category $\mathcal{Diag}$  is a functor, and  it coincides with $\hat{\otimes}$ on $\mathcal{Diag}\,\times \,\mathcal{Diag}$.  That is, it can be defined as the monoidal product  in  $\mathcal{Diag}$.
\begin{prop}
	Let $W\otimes W,\,Z\otimes Z,\,W'\otimes W',\,Z'\otimes Z'  \in \operatorname{Ob}(\mathcal{Diag})$ of dimensions $n^2$,  $m^2$,  $p^2$,  $q^2$ respectively.   Let $c:\,W\otimes W\rightarrow W'\otimes W'$, $d:\,Z\otimes Z \rightarrow Z' \otimes Z'$ be   morphisms in $\mathcal{Diag}$. 
	Let $c\boxtimes d$ denote the linear transformation obtained from the Tracy-Singh product of the representing matrices of $c$ and $d$, where $c$ and $d$  are  partitioned into blocks of the same size $n\times p$ and $m\times q$, respectively. 
	\begin{gather*}
	\textrm{Let}\;\, F_{\boxtimes}: \,\mathcal{Diag}\,\times \,\mathcal{Diag}\,\rightarrow \,\mathcal{Diag}\,\; \textrm{be defined by:}\\
	(W\otimes W,Z\otimes Z)\,\mapsto\, (W\otimes Z)\,\otimes\,(W\otimes Z)\\
	(c,d)\,\mapsto\, c\boxtimes \,d 
	\end{gather*}
	Then $F_{\boxtimes}$ is a functor equal to the restriction of $\hat{\otimes}$ on $\mathcal{Diag}\,\times \,\mathcal{Diag}$.  Furthermore, in the special case that $\operatorname{dim}(W)\,=\,\operatorname{dim}(W')$, $\operatorname{dim}(Z)\,=\,\operatorname{dim}(Z')$,   if $c$ and $d$ are $R$-matrices, then $F_{\boxtimes}(c,d)$ is also a $R$-matrix,  and if $c$ and $d$ are unitary matrices, then $F_{\boxtimes}(c,d)$ is also a unitary matrix.  
\end{prop}
\begin{proof}
	First, from Equation (\ref{eqn-defn-hat-tensor}), 
	the  action of  $F_{\boxtimes}$ on $\operatorname{Ob}(\mathcal{Diag}\times \mathcal{Diag})$ clearly coincides with that of $\hat{\otimes}$. Next,  it is enough to  show that $\hat{\otimes}$ and $F_{\boxtimes}$ coincide in $\operatorname{Mor}(\mathcal{Diag}\times \mathcal{Diag})$  on morphisms of the form $(f_1\otimes g_1, f_2\otimes g_2)$, where $f_1,g_1:W\rightarrow W'$ and $f_2,g_2:Z\rightarrow Z'$, since $c=\,\sum\limits_{i}^{}\,f^i_1\otimes g^i_1$ and $d=\,\sum\limits_{j}^{}f^j_2\otimes g^j_2$, with  $f^i_1,g^i_1:W\rightarrow W'$ and $f^j_2,g^j_2:Z\rightarrow Z'$. By  substituting  in  (\ref{eqn-defn-hat-tensor}), $U_1=V_1=W$, $U_2=V_2=Z$, $U'_1=V'_1=W'$, $U'_2=V'_2=Z'$, $f_1,g_1:W\rightarrow W'$ and $f_2,g_2:Z\rightarrow Z'$,  we have 
	$(f_1\otimes g_1)\,\hat{\otimes}\,(f_2\otimes g_2) \,=\,(f_1\,\otimes f_2) \, \otimes     \,(g_1\,\otimes\,g_2)$, a linear map  that   sends $w\otimes z \otimes w_.\otimes z_.$  to  $f_1(w)\otimes f_2(z) \otimes g_1(w_.)\otimes g_2(z_.)$,  for every $w,w_. \in W$ and $z,z_.\in Z$. From  Remark  \ref{rem-ts-general-matrices},  the linear map $(f_1\otimes g_1)\,\boxtimes\,(f_2\otimes g_2)$ sends $w\otimes z \otimes w_.\otimes z_.$  to 	$(F_{23}\,((f_1\otimes g_1)\,\otimes\,(f_2\otimes g_2))\,F_{23})\,(w\otimes z \otimes w_.\otimes z_.)\,=\, ...\,=f_1(w)\otimes f_2(z) \otimes g_1(w_.)\otimes g_2(z_.)$.  That is, $F_{\boxtimes}$ is  equal to the restriction of $\hat{\otimes}$ on $\mathcal{Diag}\,\times \,\mathcal{Diag}$. The other properties are derived  from Theorem \ref{theo-tracy-unitary}.		
\end{proof}
\part{Quantum computation and construction of entangling $2$-qudit gates}
\section{Preliminaries on quantum computing and quantum entanglement}
We follow the presentation  from the reference books on the topic \cite{book-quantum1},  \cite{book-quantum2}  and the papers  \cite{bryl}, \cite{kauf-lo1,kauf-lo1',kauf-lo4,kauf-lo3,kauf-meh} and  we refer to  these references  and the vast literature for more details. 
\begin{defn}\label{defn-qubit}
	Let $\mathbb{C}^2$ be the two-dimensional Hilbert space with two orthonormal state vectors, denoted by  $\mid 0\rangle$ and $\mid 1\rangle$, that  form a basis in bijection with the standard basis $\{(1,0),(0,1)\}$.
	A \emph{qubit (or quantum bit)} is a  state vector in $\mathbb{C}^2$
	\begin{equation}\label{eqn-qubit}
	\mid \phi \rangle=\, \alpha\,\mid 0\rangle+\beta\,\mid 1\rangle 
	\end{equation}
	where $\alpha, \beta \in \mathbb{C}$ and $\mid\alpha\mid^2 +\mid\beta\mid^2 =1$.
	We say that any linear combination of the form (\ref{eqn-qubit}) is a \emph{superposition of the states $\mid 0\rangle$ and $\mid 1\rangle$},  with \emph{amplitude} $\alpha$ for the state  $\mid 0\rangle$ and $\beta$ for the state $\mid 1\rangle$. 
\end{defn}
Intuitively, the states $\mid 0\rangle$ and $\mid 1\rangle$ are analogous to the two values $0$ and $1$ which a bit may take. The way a qubit differs from a bit is that superposition of these two states, of the form (\ref{eqn-qubit}), can also exist, in which it is not possible to say that the qubit is definitely in the state $\mid 0\rangle$ or  definitely in the state  $\mid 1\rangle$. A measurement of a qubit $\mid \phi \rangle=\alpha\,\mid 0\rangle+\beta\,\mid 1\rangle$ provides as output the bit $0$ with probability $\mid\alpha\mid^2 $ and the bit $1$ with probability $\mid\beta\mid^2 $ and the state $\mid \phi \rangle$ collapses to $\mid 0 \rangle$ or $\mid 1 \rangle$. After the measurement, all the information about the superposition is irreversibly lost. Examples of qubits  include the spin of the electron in which the two basis states are  spin up and spin down,  and in this case the basis is denoted by $\{\,\mid \uparrow\rangle\,,\,\mid \downarrow\rangle\,\}$; or the polarization of a single photon in which the two basis  states are  vertical and horizontal,   and in this case the basis is denoted by $\{\,\mid \rightarrow\rangle\,,\,\mid \leftarrow\rangle\,\}$.

More generally, a \emph{$n$-qubit} is a  state vector in the  $2^n$-dimensional Hilbert space, with an orthonormal basis  $\{\,\mid \psi_i\rangle\;\mid 1 \leq i \leq 2^n\}$ in bijection with the standard basis, of the form
\begin{equation}\label{eqn-n-qubit}
\mid \phi \rangle=\, \sum\limits_{i=1}^{i=2^n}\alpha_i\,\mid \psi_i\rangle
\end{equation}
where $\alpha_i \in \mathbb{C}$, $ 1 \leq i \leq 2^n$,  and $\sum\limits_{i=1}^{i=2^n}\mid\alpha_i\mid^2 =1$. As an example, a two-qubit has the form \begin{equation}\label{eqn-2-qubit}
\mid \phi \rangle=\, \alpha_{00}\,\mid 00\rangle+\alpha_{01}\,\mid 01\rangle+\alpha_{10}\,\mid 10\rangle+\alpha_{11}\,\mid 11\rangle
\end{equation}
In analogy with  the case of a qubit, a measurement of a $n$-qubit  of the form (\ref{eqn-n-qubit}) gives as outcome $n$ bits, each $n$-tuple of bits with  a precalculated probability. Moreover,  the state  $\mid \phi \rangle$, in the superposition of the  $2^n$ basis states,   collapses to just one of the basis states.  
\begin{defn}\label{defn-qudit}
	Let $\mathbb{C}^d$ be the d-dimensional Hilbert space with orthonormal base denoted by  $\mid 0\rangle$, $\mid 1\rangle$, ..., and $\mid d-1\rangle$.
	A \emph{qudit} is a  state vector in $\mathbb{C}^d$, 	where $\alpha_i \in \mathbb{C}$ and $\sum\limits_{i=1}^{i=d}\mid\alpha_i\mid^2  =1$:
	\begin{equation}\label{eqn-qudit}
	\mid \phi \rangle=\, \alpha_1\,\mid 0\rangle+...+\alpha_d\,\mid d-1\rangle
	\end{equation}
	Any linear combination of the form (\ref{eqn-qudit}) is a \emph{superposition of the states $\mid 0\rangle$,...,  $\mid d- 1\rangle$}.
\end{defn}
A qubit is a special case of a qudit, for the case  $d=2$. A \emph{$n$-qudit} is a state vector in the Hilbert space $(\mathbb{C}^d)^{\otimes n}$.
A quantum system with one  state vector $\mid \phi \rangle$ is called a \emph{pure state}. However, it is also possible for a system to have a set of potential different state vectors. As an example, there may be a  probability  $\frac{1}{2}$ that the state vector is 
$\mid \phi \rangle$ and a  probability  $\frac{1}{2}$ that the state vector is 
$\mid \psi \rangle$. This system is said to  be in a \emph{mixed state}. There exists a   matrix called density matrix which trace value determines whether a system is in a pure or a mixed state \cite[p.99]{book-quantum1}.
\begin{defn}
	A $n$-qudit $\mid \phi \rangle$ is   \emph{decomposable} if $\mid \phi \rangle\,=\,\mid \phi_1 \rangle\,\otimes\,\mid \phi_2 \rangle\,\otimes\,...\otimes\,\mid \phi_n \rangle$, where 
	$\mid \phi_i \rangle\,\in \, \mathbb{C}^d$, for $ 1 \leq i \leq n$.  Otherwise, $\mid \phi \rangle$ is  \emph{entangled}.  
\end{defn} 
As an example, $\mid \phi^+ \rangle \,=\, \frac{1}{\sqrt{2}}\,(\,\mid 00\rangle\,+\,\mid 11\rangle\,)$ is an entangled two-qubit,  since it cannot be decomposed as a tensor product of two qubits. In general, a two-qubit pure state $\mid \phi \rangle=\, \alpha_{00}\,\mid 00\rangle+\alpha_{01}\,\mid 01\rangle+\alpha_{10}\,\mid 10\rangle+\alpha_{11}\,\mid 11\rangle$ is entangled if and only if $\alpha_{00}\alpha_{11}-\alpha_{01}\alpha_{10}\,\neq 0$. In \cite{kauf-lo4}, the authors give a combinatorial criteria to determine whether a $n$-qubit is entangled or not.

Any quantum evolution of a  $n$-qudit, or any quantum operation on a $n$-qudit is described by a unitary operator or  square unitary matrix of size $d^n$, called a \emph{(quantum)  $n$-qudit  gate}, which transforms $\mid \phi \rangle\,=\, \sum\limits_{i=1}^{i=d^n}\alpha_i\,\mid \psi_i\rangle$ into another linear combination of  the (standard) basis states $\{\, \mid \psi_i\rangle\;\mid 1 \leq i \leq d^n\}$.
\begin{ex}\label{ex-c-ybe-2qubits}
	Let $c$ be the unitary $R$-matrix from Example \ref
	{ex1-kaufman}. So, $c$  is a 2-qubit gate and it   acts on the basis states  $\{\,\mid 00\rangle,\,\mid 01\rangle,\,\mid 10\rangle,\,\mid 11\rangle\,\}$ of $(\mathbb{C}^2)^{ \otimes 2}$,  with $\mid 00\rangle\,=\,\mid 0\rangle\,\otimes\,\mid 0\rangle$, $\mid 01\rangle\,=\,\mid 0\rangle\,\otimes\,\mid 1\rangle$, and so on, such  that $c\,\mid 00\rangle\,=\,\frac{1}{\sqrt{2}}\mid 00\rangle\,-\,\frac{1}{\sqrt{2}}\,\mid 11\rangle$, $c\,\mid 01\rangle\,=\,\frac{1}{\sqrt{2}}\mid 01\rangle\,+\,\frac{1}{\sqrt{2}}\,\mid 10\rangle$,  $c\,\mid 10\rangle\,=\,\,-\frac{1}{\sqrt{2}}\mid 01\rangle\,+\,\frac{1}{\sqrt{2}}\,\mid 10\rangle$
	and $c\,\mid 11\rangle\,=\,\frac{1}{\sqrt{2}}\mid 00\rangle\,+\,\frac{1}{\sqrt{2}}\,\mid 11\rangle$.
\end{ex}
\begin{ex}\label{ex-cnot-2qubit}
	A very important example of  2-qubit gate is the following   unitary square matrix of size $4$ which acts on $2$-qubits: 
	$CNOT=\begin{pmatrix}
	1& 0 & 0 & 	0\\
	0 & 1& 0 & 0 \\
	0 &	0 & 0& 1 \\
	0&0 & 1 &	0\\
	\end{pmatrix}$;  not an $R$-matrix.
\end{ex}
\begin{defn} \cite{bryl}
	A $2$-qudit gate $L:\,(\mathbb{C}^d)^{\otimes 2}\,\rightarrow\,(\mathbb{C}^d)^{\otimes 2}$ is \emph{primitive} if $L$ maps decomposable $2$-qudit to decomposable $2$-qudit, otherwise $L$ is  said to be \emph{imprimitive}.
\end{defn}
In other words, a $2$-qudit gate  $L$ is said to be \emph{imprimitive}, if there exists a decomposable $2$-qudit $\mid \phi \rangle$ such that $L\,\mid \phi \rangle$ is an entangled $2$-qudit. An imprimitive $2$-qudit gate is often called \emph{entangling}, as in \cite{kauf-lo3}. There is a criteria to determine whether a $2$-qudit is primitive.
\begin{thm}\cite{bryl}\label{thm-bryl-criteria}
	Let $P:\,(\mathbb{C}^d)^{\otimes 2}\,\rightarrow\,(\mathbb{C}^d)^{\otimes 2}$ denote the swap gate, that is the $2$-qudit gate such that $P\mid \alpha\beta \rangle\,=\,\mid \beta\alpha\rangle$.  Let  $L:\,(\mathbb{C}^d)^{\otimes 2}\,\rightarrow\,(\mathbb{C}^d)^{\otimes 2}$ be a  $2$-qudit gate. Then $L$  is primitive if  and only if $L=\,L_1\otimes L_2$ or  $L=\,(L_1\otimes L_2)\,P$,  for some $1$-qudit gates $L_1$, $L_2$.
\end{thm}
In \cite{kauf-lo3}, the authors answer  the question of which  $2$-qubit gates which satisfy   the YBE are entangling, using the classification from \cite{dye} and the criteria for entanglement from \cite{bryl}. The $2$-qubit gates from Examples  \ref{ex-c-ybe-2qubits} and \ref{ex-cnot-2qubit} are both entangling operators \cite{book-quantum2,kauf-lo3,book-quantum1}. As Example \ref{ex-set-theoretic-non-entangling}  illustrates it, not every $2$-qubit gate that is an $R$-matrix is entangling.
\begin{ex}\label{ex-set-theoretic-non-entangling}
	Let $c_{2,1}$ and $c_{2,2}$ be the two $R$-matrices of size 4 corresponding to the two set-theoretic solutions of the YBE,  from Example \ref{ex-permut-sol-2}. It is easy to show that both $2$-qubit gates $c_{2,1}$ and $c_{2,2}$ are primitive. Let
	$\mid \phi \rangle=\, \alpha_{00}\,\mid 00\rangle+\alpha_{01}\,\mid 01\rangle+\alpha_{10}\,\mid 10\rangle+\alpha_{11}\,\mid 11\rangle$ be a  decomposable two-qubit pure state, that is  $\alpha_{00}\alpha_{11}-\alpha_{01}\alpha_{10}\,= 0$.  Then 
	$c_{2,1} \,\mid \phi \rangle=\, \alpha_{00}\,\mid 00\rangle+\alpha_{10}\,\mid 01\rangle+\alpha_{01}\,\mid 10\rangle+\alpha_{11}\,\mid 11\rangle$ and  $\alpha_{00}\alpha_{11}-\alpha_{10}\alpha_{01}\,= 0$, that is  $c_{2,1}\,\mid \phi \rangle$ is decomposable. The same argument holds for $c_{2,2}$.
\end{ex}

\section{Construction of entangling and decomposable  $2$-qudit gates}
Let $c:\,(\mathbb{C}^n)^{\otimes 2}\,\rightarrow\,(\mathbb{C}^n)^{\otimes 2}$ and $d:\,(\mathbb{C}^m)^{\otimes 2}\,\rightarrow\,(\mathbb{C}^m)^{\otimes 2}$  be $2$-qudit gates. Let $c \boxtimes d$ denote their  Tracy-Singh product,  where  $c$ and $d$ have  a  canonical  block partition as in (\ref{eqn-c-blocks}). In this section, we show  that if either $c$ or  $d$ is an  entangling $2$-qudit gate, then  $c \boxtimes d:\,(\mathbb{C}^{nm})^{\otimes 2}\,\rightarrow\,(\mathbb{C}^{nm})^{\otimes 2}$ is  also an entangling $2$-qudit gate.
Let $\mid\phi\rangle$ and $\mid \psi\rangle$  be  $2$-qudits in $(\mathbb{C}^{n})^{\otimes 2}$ and $(\mathbb{C}^{m})^{\otimes 2}$,    respectively. Note that   $\mid\phi\rangle$ and $\mid \psi\rangle$   can be considered as  column vectors in  $(\mathbb{C})^{n^2}$ and $(\mathbb{C})^{m^2}$, denoted by   $\phi$ and  $\psi$ respectively, and as such one can compute   $\phi\boxtimes\psi$, with  $\phi$ partitioned into $n$ blocks of length $n$, $\phi_{11},...,\phi_{n1}$	 and  $\psi$ partitioned  into $m$ blocks of length $m$, $\psi_{11},...,\psi_{m1}$, as in Equation (\ref{eqn-partition-phi-psi}).  So,  depending on the kind of computations,   we use either terminology  for $\phi$ and $\psi$.  In the following lemmas, we assume all the above assumptions.
\begin{equation}\label{eqn-partition-phi-psi}
\scalemath{0.9}{
	\phi \,=\,
	\begin{pmatrix}
	\phi_{11}	\\
	\hline
	... \\
	...\\
	\hline
	\phi_{n1}		\\
	\end{pmatrix} \;\;\;\;\; \;\;\;\;\;\textrm{and}\;\;\;\;\;\;\;\;\;\;
	\begin{pmatrix}
	\psi_{11}	\\
	\hline
	... \\
	...\\
	\hline
	\psi_{m1}		
	\end{pmatrix}}
\end{equation}
\begin{lem}\label{lem-dist-phipsi}
	Under the above assumptions, the following holds:
	\begin{enumerate}[(i)]
		\item  \label{eqn-distr-c-phi}
		$	\scalemath{1}{
			(c \boxtimes d)\,(\phi \boxtimes \psi)\,=\,c\phi \boxtimes d\psi}$, 
		where  $c\phi$ is partitioned into $n$ blocks of length $n$ and  $d\psi$ into $m$ blocks of length $m$. 
		\item  If  $\phi\,=\, \alpha \otimes \beta$ and $\psi\,=\, \gamma \otimes \delta$, where $\alpha, \beta \in \mathbb{C}^{n}$ and $\gamma, \delta \in \mathbb{C}^{m}$, then
		\begin{equation}\label{eqn-tracy-decomposable}
		\phi \boxtimes \psi\,=\, (\alpha \otimes \gamma)\;\otimes\;(\beta \otimes \delta)
		\end{equation}
		\item If  $\phi$ and $\psi$ are unit vectors, then  $\phi \boxtimes\psi$ is also a unit vector.
	\end{enumerate}
\end{lem}
\begin{proof}
	$(i)$  results from Theorem \ref{thm-tracy} $(v)$, since  the products $c\phi$ and $d\psi$ are well-defined.  \\
	$(ii)$ Assume $\alpha=(\alpha_1,...,\alpha_n)^t$,  $\gamma=(\gamma_1,...,\gamma_m)^t$. So, 
	$\phi=\,(\alpha_1\beta, ...,\alpha_n\beta)^t$ and $\psi=\, (\gamma_1\delta,...,\gamma_m\delta)^t$.
	\begin{figure*}[h!]
		\begin{equation*}\label{eqn-phi-tracy-psi}
		\phi \boxtimes \psi=
		\begin{pmatrix}
		\alpha_1\beta	\\
		\hline
		... \\
		...\\
		\hline
		\alpha_n\beta		\\
		\end{pmatrix} \boxtimes
		\begin{pmatrix}
		\gamma_1\delta		\\
		\hline
		... \\
		...\\
		\hline
		\gamma_m\delta			\\
		\end{pmatrix}\,=\,
		\begin{pmatrix}
		\alpha_1\beta\otimes 	\gamma_1\delta		\\
		... \\
		\alpha_1\beta \otimes \gamma_m\delta\\
		\hdashline[2pt/2pt]
		... \\
		\hdashline[2pt/2pt]
		\alpha_n\beta	\otimes \gamma_1\delta		\\
		... \\
		\alpha_n\beta\otimes\gamma_m\delta\\
		\end{pmatrix}\,=\,
		\begin{pmatrix}
		\alpha_1\gamma_1(\beta\otimes 	\delta	)	\\
		... \\
		\alpha_1\gamma_m(\beta \otimes \delta)\\
		\hdashline[2pt/2pt]
		... \\
		\hdashline[2pt/2pt]
		\alpha_n\gamma_1(\beta	\otimes \delta)		\\
		... \\
		\alpha_n\gamma_m(\beta\otimes\delta)\\
		\end{pmatrix}=\, (\alpha \otimes \gamma)\;\otimes\;(\beta \otimes \delta)
		\end{equation*}
	\end{figure*}\\
	$(iii)$ If  $\phi$ and $\psi$ are unit vectors, that is $\phi\,\phi^*=\,1$ and $\psi\,\psi^*=\,1$, then  $(\phi \boxtimes\psi)\,(\phi \boxtimes\psi)^*= \phi\,\phi^*\boxtimes \psi\,\psi^*=\,1 $, since  $\phi\,\phi^*$ and $\psi\,\psi^*$ are well-defined. So,   $\phi \boxtimes\psi$ is  also a  unit vector.
\end{proof}
A direct proof of  Equation (\ref{eqn-distr-c-phi}) relies on the definition of the Tracy-Singh product of matrices and the  mixed-product property of the Kronecker product of  matrices, and all the above  assumptions on the partitions are  very crucial.
\begin{lem}\label{lem-decomp-iff}
	Under the above assumptions,   $\phi \boxtimes \psi$ is a  decomposable  $2$-qudit in $(\mathbb{C}^{nm})^{\otimes 2}$  if and only if   $\phi$ and $\psi$ are decomposable.
\end{lem}
\begin{proof}
	Assume 	that $\phi$ and $\psi$ are decomposable, with $\phi\,=\, \alpha \otimes \beta$ and $\psi\,=\, \gamma \otimes \delta$, where $\alpha, \beta \in \mathbb{C}^{n}$ and $\gamma, \delta \in \mathbb{C}^{m}$, then from Lemma  \ref{lem-dist-phipsi}$(ii)$, 
	$\phi \boxtimes \psi\,=\, (\alpha \otimes \gamma)\;\otimes\;(\beta \otimes \delta)$, with $\alpha \otimes \gamma\in \mathbb{C}^{nm}$ and   $\beta \otimes \delta \in \mathbb{C}^{nm}$. That is,  $\phi \boxtimes \psi$ is a  decomposable $2$-qudit in $(\mathbb{C}^{nm})^{\otimes 2}$. \\
	Assume that $\phi \boxtimes \psi$ is decomposable, that  is there exist $\sigma=\,(\sigma_1,...,\sigma_{nm})^t$ and $\tau=\,(\tau_1,...,\tau_{nm})^t$ in $\mathbb{C}^{nm}$, such that $\phi \boxtimes \psi=\,\sigma \otimes  \tau$.  
	Let denote by $\phi=\,(\phi_{11},...,\phi_{n1})^t$ and $\psi=\,(\psi_{11},...,\psi_{m1})^t$ the  partition of $\phi$ and $\psi$ into blocks of size $n$ and $m$ respectively.  So, from the computations of   $\phi \boxtimes \psi$ and  $\sigma \otimes  \tau$, and the equality of each block of size $nm$,  as described in   Equation (\ref{eqn-comput-phi-box-psi}), we have  $\tau=\frac{1}{\sigma_1}	\,\phi_{11}\otimes 		\psi_{11}=...=	\frac{1}{\sigma_m}\,\phi_{11}\otimes 		\psi_{m1}$, ...,  and 
	$\tau=\frac{1}{\sigma_*}\,	\phi_{n1}\otimes 		\psi_{11}=...=	\frac{1}{\sigma_{nm}}\,\phi_{n1}\otimes 		\psi_{m1}$, whenever the denominators are not zero. 
	\begin{equation}\label{eqn-comput-phi-box-psi}
	\scalemath{0.8}{
		\phi \boxtimes \psi=
		\begin{pmatrix}
		\phi_{11}	\\
		\hline
		... \\
		...\\
		\hline
		\phi_{n1}		\\
		\end{pmatrix} \boxtimes
		\begin{pmatrix}
		\psi_{11}	\\
		\hline
		... \\
		...\\
		\hline
		\psi_{m1}		\\
		\end{pmatrix}\,=\,
		\begin{pmatrix}
		\phi_{11}\otimes 		\psi_{11}	\\
		\hdashline[2pt/2pt]
		... \\
		\hdashline[2pt/2pt]
		\phi_{11}\otimes 		\psi_{m1}	\\
		\hdashline[0.8pt/0.8pt]
		... \\
		\hdashline[0.8pt/0.8pt]
		\phi_{n1}\otimes 		\psi_{11}	\\
		\hdashline[2pt/2pt]
		... \\
		\hdashline[2pt/2pt]
		\phi_{n1}\otimes 		\psi_{m1}	\\
		\end{pmatrix}\;\;\;\;\;\;\textrm{and}\;\;\;\;\;\;
		\sigma \otimes \tau=\,
		\begin{pmatrix}
		\sigma_1\,\tau	\\	
		\hdashline[2pt/2pt]
		... \\
		\hdashline[2pt/2pt]
		\sigma_m\,\tau	\\
		\hdashline[0.8pt/0.8pt]
		... \\
		\hdashline[0.8pt/0.8pt]
		\sigma_{*}\,\tau\\
		\hdashline[2pt/2pt]
		... \\
		\hdashline[2pt/2pt]
		\sigma_{nm}\,\tau	
		\end{pmatrix}
	}
	\end{equation}
	
	So, there exist scalars $\mu_2,...,\mu_n,\eta_2,...,\eta_m$ (fractions of the $\sigma_i$'s) such that 
	$\phi_{21}=\mu_2\,\phi_{11},...,\phi_{n1}=\mu_n\,\phi_{11}$
	and 
	$\psi_{21}=\eta_2\,\psi_{11},...,\psi_{m1}=\eta_m\,\psi_{11}$. That is, $\phi $ and $\psi$ are both decomposable, as described in Equation (\ref{Eqn-form-decomp-phi-psi}).
	\begin{equation}\label{Eqn-form-decomp-phi-psi}
	\phi =
	\begin{pmatrix}
	\phi_{11}	\\
	\hline
	\mu_2\,		\phi_{11}	\\
	\hline
	...\\
	\hline
	\mu_n\,	\phi_{11}		\\
	\end{pmatrix} \;=
	\begin{pmatrix}
	1	\\
	\mu_2	\\
	...\\
	\mu_n	\\
	\end{pmatrix} \otimes \phi_{11}\;\;\;\;\;\;\textrm{and}\;\;\;\;\;\;
	\psi=
	\begin{pmatrix}
	\psi_{11}	\\
	\hline
	\eta_2\,\psi_{11}	\\
	\hline
	...\\
	\hline
	\eta_m\,\psi_{11}		\\
	\end{pmatrix}\;=
	\begin{pmatrix}
	1	\\
	\eta_2	\\
	...\\
	\eta_m	\\
	\end{pmatrix} \otimes \psi_{11}
	\end{equation}
\end{proof}
Lemma \ref{lem-decomp-iff} implies that  if $\phi$ or  $\psi$  is an entangled   $2$-qudit, then $\phi \boxtimes \psi$  is also entangled,  where  $\phi$ is partitioned into $n$ blocks of length $n$ and  $\psi$ into $m$ blocks of length $m$. This enables us to prove that the Tracy-Singh product of two $2$-qudit gates where one of them is an entangling gate is also entangling.  We reformulate Theorem \ref{thm2} more precisely as follows and prove it.
\begin{thm}\label{theo-tracy-entangling-also-ent}
	Let $c:\,(\mathbb{C}^n)^{\otimes 2}\,\rightarrow\,(\mathbb{C}^n)^{\otimes 2}$ and $d:\,(\mathbb{C}^m)^{\otimes 2}\,\rightarrow\,(\mathbb{C}^m)^{\otimes 2}$  be $2$-qudit gates. Let $c \boxtimes d$ denote their  Tracy-Singh product,  where  $c$ and $d$ have  a   block partition as in (\ref{eqn-c-blocks}). 
	\begin{enumerate}[(i)]
		\item  Assume  $c$ is entangling.
		Then $c \boxtimes d:\,(\mathbb{C}^{nm})^{\otimes 2}\,\rightarrow\,(\mathbb{C}^{nm})^{\otimes 2}$ is  also  entangling.
		\item  Assume  $\phi$  is  a decomposable    $2$-qudit in $(\mathbb{C}^{n})^{\otimes 2}$ such that $c \phi$ is an  entangled $2$-qudit in $(\mathbb{C}^{n})^{\otimes 2}$.  Let $\psi$ be  any  decomposable   $2$-qudit in  $(\mathbb{C}^{m})^{\otimes 2}$. Then $\phi \boxtimes \psi$, with partition as in (\ref{eqn-partition-phi-psi}),  is a decomposable   $2$-qudit in  $(\mathbb{C}^{nm})^{\otimes 2}$ such that  $(c \boxtimes d)\,(\phi \boxtimes \psi)$
		is an entangled $2$-qudit in  $(\mathbb{C}^{nm})^{\otimes 2}$.
		\item  Assume  $c$ and $d$ are primitive. If $c\,=\,c_1\otimes c_2$ and  $d\,=\,d_1\otimes d_2$ or $c\,=\,(c_1\otimes c_2)S_c$ and  $d\,=\,(d_1\otimes d_2)S_d$, where $S_c$ and $S_d$ are  the  swap maps on $(\mathbb{C}^n)^{\otimes 2}$ and $(\mathbb{C}^m)^{\otimes 2}$ respectively.
		Then $c \boxtimes d:\,(\mathbb{C}^{nm})^{\otimes 2}\,\rightarrow\,(\mathbb{C}^{nm})^{\otimes 2}$ is  also primitive.
	\end{enumerate} 
\end{thm}
\begin{proof}
	$(i)$, $(ii)$ From Lemma \ref{lem-decomp-iff}, $\phi \boxtimes \psi$ is a decomposable   $2$-qudit in  $(\mathbb{C}^{nm})^{\otimes 2}$. It remains to prove that $(c \boxtimes d)\,(\phi \boxtimes \psi)$
	is an entangled $2$-qudit in  $(\mathbb{C}^{nm})^{\otimes 2}$. From Lemma \ref{lem-dist-phipsi}$(i)$, $(c \boxtimes d)\,(\phi \boxtimes \psi)=(c  \phi )\boxtimes (d \psi)$. As $c  \phi $ is  entangled, $(c  \phi )\boxtimes (d \psi)$ is also entangled from  Lemma \ref{lem-decomp-iff}. That is, $(c \boxtimes d)\,(\phi \boxtimes \psi)$	is an entangled $2$-qudit in  $(\mathbb{C}^{nm})^{\otimes 2}$ and 
	$c \boxtimes d:\,(\mathbb{C}^{nm})^{\otimes 2}\,\rightarrow\,(\mathbb{C}^{nm})^{\otimes 2}$ is  an entangling $2$-qudit gate. Note that if we assumed that $d$ is entangling instead of $c$, then we  would choose $\psi$ to  be a decomposable    $2$-qudit in $(\mathbb{C}^{m})^{\otimes 2}$ such that $d \psi$ is an  entangled $2$-qudit in $(\mathbb{C}^{m})^{\otimes 2}$ and exactly the same proof would hold.\\
	$(iii)$ From Remark \ref{rem-ts-general-matrices}, we have $c \boxtimes d\,=\,(F_{23}\,(c \otimes d)\,F_{23})\,(x_1\otimes y_1\otimes x_2\otimes y_2)$, for every $x_1,x_2 \in \mathbb{C}^{n}$,  $y_1,y_2 \in \mathbb{C}^{m}$. In the first case, this gives $(c_1(x_1)\,\otimes \,d_1(y_1)\,\otimes\,c_2( x_2)\,\otimes\, d_2(y_2))$, that is 
	$c \boxtimes d\,=\,(c_1\otimes d_1)\,\otimes \,(c_2\otimes d_2)$, and  $c \boxtimes d$ is primitive from Theorem \ref{thm-bryl-criteria}. In the second case, this gives $c_1(x_2)\,\otimes \,d_1(y_2)\,\otimes\,c_2( x_1)\,\otimes\, d_2(y_1)$, that is 
	$c \boxtimes d\,=\,((c_1\otimes d_1)\,\otimes \,(c_2\otimes d_2))\,P$, where $P$ is the swap map on $(\mathbb{C}^{nm})^{\otimes 2}$, and  $c \boxtimes d$ is primitive from Theorem \ref{thm-bryl-criteria}.  
\end{proof}
It is not clear whether it is always true that if $c$ and $d$ are primitive  $2$-qudit gates, then $c \boxtimes d:\,(\mathbb{C}^{nm})^{\otimes 2}\,\rightarrow\,(\mathbb{C}^{nm})^{\otimes 2}$ is  a primitive $2$-qudit gate.  From Theorem \ref{theo-tracy-entangling-also-ent}, one can construct infinitely many families of entangling $2$-qudit gates.  Indeed, one can take the Tracy-Singh  product of the CNOT gate  (see Example \ref{ex-cnot-2qubit}), which is known to be (strongly) entangling, with any arbitrary   $2$-qudit gate and obtain a new entangling $2$-qudit gate. The same process can be iterated to obtain infinitely many entangling $2$-qudit gates. A question that arises naturally is whether there exists a simple way to decompose a $2$-qudit gate of the form  $c \boxtimes d$ as a product of CNOT and single qudit gates.

We now turn to  the proof of Theorem \ref{thm0}. For that, we need the following proposition.
\begin{prop}\label{prop-square-free-3-entangling}
	\begin{enumerate}[(i)]
		\item  Let $p\geq 3$ be a prime. Let $r:\,(\mathbb{C}^p)^{\otimes 2}\,\rightarrow\,(\mathbb{C}^p)^{\otimes 2}$ be the   $2$-qudit gate corresponding to the square-free, non-degenerate involutive st-YBE defined by:
		\begin{gather*}
		\sigma_1=...=\sigma_{p-1}\,=\,	\gamma_1=...=\gamma_{p-1}\,=\,Id\\
		\sigma_{p}=\gamma_p\,=\,(1,2)\cdot...\cdot(p-2,p-1)
		\end{gather*}
		Then $r$ is an  entangling $2$-qudit gate that satisfies the YBE.
		\item Let  $p\geq 2$ be a prime.  Let $s:\,(\mathbb{C}^p)^{\otimes 2}\,\rightarrow\,(\mathbb{C}^p)^{\otimes 2}$ be the   $2$-qudit gate corresponding to the  unique indecomposable  non-degenerate involutive (cyclic permutation) st-YBE:
		\begin{gather*}
		\sigma=(1,2,...,p) \;\;\;\;\; \textrm{and}\;\;\;\;\; \gamma=(p,..,2,1)
		\end{gather*}
	\end{enumerate}
	Then  $s$ is a primitive $2$-qudit gate that satisfies the YBE.
\end{prop}
\begin{proof}
	$(i)$ We use the criteria given in  Theorem \ref{thm-bryl-criteria} from \cite{bryl} to show that $r$  is entangling. Before we proceed, we describe the matrix $r$. From the definition of $r$, $r(e_i\otimes e_j)=e_j\otimes e_i$, for every $1\leq i,j\leq p-1$, $r(e_i\otimes e_p)=e_p\otimes e_{\gamma_p(i)}$, $r(e_p\otimes e_i)=e_{\sigma_p(i)}\otimes e_p$, for  every $1\leq i\leq p-1$, and  the matrix $r$ can be described with a  partition of square blocks of size $p$ as in Equation (\ref{eqn-r-square-free-general}):  
	\begin{equation}\label{eqn-r-square-free-general}
	r=\;\;\;\;\;\left(\begin{array}{c:c:c:c:c|c}
	E_{1,1}
	& E_{2,1}
	&...&E_{p-2,1}& E_{p-1,1} &E_{p,2} \\
	\hdashline
	E_{1,2}
	& E_{2,2}
	&...&E_{p-2,2}& E_{p-1,2} &E_{p,1} \\
	\hdashline
	...
	& ...
	&...& ... &...&... \\
	\hdashline
	E_{1,p-2}	
	& E_{2,p-2}
	&...&E_{p-2,p-2}& E_{p-1,p-2} &E_{p,p-1} \\
	\hdashline
	E_{1,p-1}	
	& E_{2,p-1}
	&...&E_{p-2,p-1}& E_{p-1,p-1} &E_{p,p-2} \\
	\hline
	E_{2,p}	
	& E_{1,p}
	&...&E_{p-1,p}& E_{p-2,p} &E_{p,p} \\
	\end{array}\right)	\end{equation}
	Assume by contradiction that $r=A \otimes B$, where $A,B: \,\mathbb{C}^p \,\rightarrow\,\mathbb{C}^p$. Then, from the equality of the diagonal  square blocks of size $p$ in  $r$ and in $A \otimes B$ (see (\ref{eqn-P-o})):
	\[a_{11}B=E_{1,1}\;\;\;,\;\;\;a_{22}B=E_{2,2},\;\;...\;\;,\;\;\textrm{and }\;\;a_{pp}B=E_{p,p}\]
	But this  is not possible.
	Assume by contradiction that $r=(A \otimes B)P$, where $A,B: \,\mathbb{C}^p \,\rightarrow\,\mathbb{C}^p$ and $P$ is the swap $2$-qudit gate.  The matrices $A \otimes B$ and  $P$, with a  partition of square blocks of size $p$,  are  described in Equation (\ref{eqn-P-o}).
	\begin{equation}\label{eqn-P-o}
	A\otimes B=\;\left(\begin{array}{c:c:c:c}
	
	a_{11}B
	& a_{12}B
	&... &a_{1p}B \\
	\hdashline
	a_{21}B
	& a_{22}B
	&... &a_{2p}B \\
	\hdashline
	...
	& ...
	&... &... \\
	\hdashline
	a_{p1}B
	& a_{p2}B
	&... &a_{pp}B \\
	\end{array}\right) \;\;\; \textrm{and}\;\;\;
	P=\;\left(\begin{array}{c:c:c:c}
	E_{1,1}
	& E_{2,1}
	&...& E_{p,1} \\
	\hdashline
	E_{1,2}
	& E_{2,2}
	&...& E_{p,2} \\
	\hdashline
	...
	& ...
	&... &... \\
	\hdashline
	E_{1,p}	
	& E_{2,p}
	&...& E_{p,p} \\
	\end{array}\right)
	\end{equation}
	From   the equality of the  square blocks of size $p$ at  the first column in $(A \otimes B)P$ and  in  $r$:
	\begin{equation}\label{eqn-syst-square-free}
	\begin{cases}
	a_{11}BE_{1,1}\,+\,a_{12}BE_{1,2}\,+\,...\,+\,a_{1p}BE_{1,p}\;=\;E_{1,1}\\
	a_{21}BE_{1,1}\,+\,a_{22}BE_{1,2}\,+\,...\,+\,a_{2p}BE_{1,p}\;=\;E_{1,2}\\
	...\\
	a_{p-1,1}BE_{1,1}\,+\,a_{p-1,2}BE_{1,2}\,+\,...\,+\,a_{p-1,p}BE_{1,p}\;=\;E_{1,p-1}\\
	a_{p1}BE_{1,1}\,+\,a_{p2}BE_{1,2}\,+\,...\,+\,a_{pp}BE_{1,p}\;=\;E_{2,p}\\
	\end{cases}   
	\end{equation}
	First, note that if $a_{ii}=0$ for some $1\leq i\leq p$, there would be a contradiction to the $i$-th equation in (\ref{eqn-syst-square-free}).
	From the properties of the product of a matrix $ B$ by an elementary matrix, we have, 
	from the first $p-1$  equations in  (\ref{eqn-syst-square-free}),  that the first column of  $B$ is equal to $\frac{1}{a_{11}}(1,0,...,0)^t\,=\,...\,=\frac{1}{a_{p-1,p-1}}(1,0,...,0)^t$, with $a_{11},...,a_{pp}\neq 0$  and all the other coefficients  equal to zero.  But, from the last equation in (\ref{eqn-syst-square-free}), the first column of  $B$ is equal to $\frac{1}{a_{pp}}(0,1,...,0)^t$, a contradiction. So, $r$ cannot be  equal $A \otimes B$, nor $(A \otimes B)P$, for any $A,B: \,\mathbb{C}^p \,\rightarrow\,\mathbb{C}^p$, that is $r$ is an entangling $2$-qudit gate and it  satisfies  the YBE.\\
	$(ii)$ The matrix $s$ is described in Example \ref{ex-square-free-sol-3} for the case $p=3$, and its general form can be deduced easily from it. Taking  $A$ and $B$ to be the permutation matrices of size $p$ corresponding to the permutations  $(1,2,...,p)$ and $(p,p-1,...,1)$, respectively, and $P$ the swap matrix from (\ref{eqn-P-o}), then $s=\,(A\otimes B)\,P$, that is $s$ is a primitive $2$-qudit gate and it   satisfies  the YBE.
\end{proof}
\begin{proof}[Proof of Theorem \ref{thm0}]
	In   \cite{kauf-lo3}, the authors prove that the  $R$-matrix $c$,  from Example \ref{ex1-kaufman}, 	is a (strongly) entangling $2$-qubit gate. So,   from Theorem \ref{theo-tracy-entangling-also-ent}, 
	for any  $2$-qudit gate, $d:\,(\mathbb{C}^m)^{\otimes 2}\,\rightarrow\,(\mathbb{C}^m)^{\otimes 2}$,  that satisfies  the YBE, the Tracy-Singh product of   $c$ with  $d$   is an entangling $2$-qudit gate $c \boxtimes d:\,(\mathbb{C}^{2m})^{\otimes 2}\,\rightarrow\,(\mathbb{C}^{2m})^{\otimes 2}$. Furthermore, from Theorem \ref{thm1}, $c \boxtimes d$ satisfies  the YBE. That is, 
	for every even integer $d\geq 2$,  there exists an entangling  $2$-qudit gate  $R:\,(\mathbb{C}^d)^{\otimes 2}\,\rightarrow\,(\mathbb{C}^d)^{\otimes 2}$, where  $R$ satisfies  the YBE. 
	From Proposition \ref{prop-square-free-3-entangling}, 
	the matrix $r$ corresponding to the square-free st-YBE with $\mid X \mid=p\geq 3$ describes a $2$-qudit gate $r:\,(\mathbb{C}^p)^{\otimes 2}\,\rightarrow\,(\mathbb{C}^p)^{\otimes 2}$ that is entangling. So,  as before, the Tracy-Singh product of   $r$ with  $d$   is an entangling $2$-qudit gate $r\boxtimes d:\,(\mathbb{C}^{pm})^{\otimes 2}\,\rightarrow\,(\mathbb{C}^{pm})^{\otimes 2}$ that satisfies  the YBE. That is, 
	for every integer $d\geq 2$,  there exists an entangling  $2$-qudit gate  $r:\,(\mathbb{C}^d)^{\otimes 2}\,\rightarrow\,(\mathbb{C}^d)^{\otimes 2}$, where  $R$ satisfies  the YBE .\\
	The Tracy-Singh product of  any  $2$-qudit gates $s,s'$  corresponding to the cyclic permutation solutions of prime order $p,p' \geq 2$  is a primitive  $2$-qudit gate. Indeed, from  Proposition \ref{prop-square-free-3-entangling},  both $s$ and $s'$ are primitive and of  the same form  $s\,=\,(A \otimes B)P$, $s'\,=\,(A' \otimes B')P'$, so from Theorem \ref{theo-tracy-entangling-also-ent}$(iii)$, $s \boxtimes s'$ is a primitive $2$-qudit gate and it satisfies the YBE from Theorem  \ref{thm1}. Reiterating this process, one has for every integer $d\geq 2$,  a primitive   $2$-qudit gate  that  satisfies  the YBE .
\end{proof}

\bigskip\bigskip\noindent
{ Fabienne Chouraqui}\\
\smallskip\noindent
University of Haifa at Oranim, Israel.

\smallskip\noindent
E-mail: {\tt fabienne.chouraqui@gmail.com} ;
                {\tt fchoura@sci.haifa.ac.il}
\end{document}